\documentclass[11pt,a4paper]{amsart}

\usepackage[left=3cm, top=3cm,bottom=3cm,right=3cm]{geometry}

\usepackage{mathtools,amssymb,amsthm,mathrsfs,calc,graphicx,stmaryrd,xcolor,dsfont,pgfplots,bbm,float,array,epsfig,color}
\usepackage[numbers,square]{natbib}
\usepackage[british]{babel}
\usepackage{amsfonts,amsmath,amsthm,amssymb}
\usepackage{verbatim}
\usepackage{multirow}
\usepackage{tikz}
\usepackage{bbm}

\numberwithin{equation}{section}
\numberwithin{figure}{section}

\newtheorem{thm}{Theorem}[section]
\newtheorem{lem}[thm]{Lemma}

\newtheorem{prop}[thm]{Proposition}

\theoremstyle{remark}

\theoremstyle{definition}

\numberwithin{equation}{section}

\def\P{{\mathbb{P}}}
\def\Z{{\mathbb{Z}}}
\def\mz{{\mathbb{Z}}}
\def\R{{\mathbb{R}}}

\def\F{{\mathcal{F}}}

\newcommand{\od}{\overset{{\rm d}}{=}}
\newcommand{\dod}{\overset{{\rm d}}{\longrightarrow}}
\newcommand{\topr}{\overset{\mathbb{P}}{\longrightarrow}}
\newcommand{\tov}{\overset{{\rm v}}{\longrightarrow}}

\newcommand{\PP}{\mathbb{P}}
\newcommand{\E}{\mathbb{E}}
\newcommand{\me}{\mathbb{E}}

\newcommand{\G}{\mathcal{G}}
\newcommand{\X}{\mathcal{X}}
\renewcommand{\P}{\mathbb{P}}
\newcommand{\NN}{\mathbb{N}}
\newcommand{\N}{\mathbb{N}}

\newcommand{\ZZ}{\mathbb{Z}}
\newcommand{\W}{\mathbb{W}}
\newcommand{\Zc}{Z^{\rm crit}}
\newcommand{\Wc}{W^{\rm crit}}

\newcommand{\8}{\infty}

\renewcommand{\a}{\alpha}

\renewcommand{\l}{\lambda}

\begin{document}

\title[Random walks in a strongly sparse random environment]{Random walks in a strongly sparse random environment}

\author[D. Buraczewski, P. Dyszewski, A. Iksanov and A. Marynych]{Dariusz Buraczewski, Piotr Dyszewski, Alexander Iksanov and~Alexander~Marynych}

\date{}

\begin{abstract}
The integer points (sites) of the real line are marked by the
positions of a standard random walk. We say that the set of marked
sites is weakly, moderately or strongly sparse depending on
whether the jumps of the standard random walk are supported by a
bounded set, have finite or infinite mean, respectively. Focussing
on the case of strong sparsity we consider a nearest neighbor
random walk on the set of integers having jumps $\pm 1$ with
probability $1/2$ at every nonmarked site, whereas a random drift
is imposed at every marked site. We prove new distributional limit
theorems for the so defined random walk in a strongly sparse
random environment, thereby complementing results obtained
recently in Buraczewski et al. (2018+) for the case of moderate
sparsity and in Matzavinos et al. (2016) for the case of weak
sparsity. While the random walk in a strongly sparse random
environment exhibits either the diffusive scaling inherent to a
simple symmetric random walk or a wide range of subdiffusive
scalings, the corresponding limit distributions are non-stable.
\end{abstract}

\keywords{Branching process in a random environment with
immigration, convergence in distribution, random walk in a random
environment, sparse random environment} \subjclass[2010]{Primary:
60K37; Secondary: 60F05, 60F15, 60J80.}

\maketitle

\section{Introduction}\label{intro}

A simple random walk (SRW) on the set $\Z$ of integers is one of
the most fundamental objects in both classical and modern
probability. We consider a slightly perturbed version of SRW
obtained by imposing a random drift at some randomly chosen
(marked) sites. Allowance is made for occasional huge gaps between
the marked sites which thus form a rather sparse subset of $\Z$.
Our main purpose is to reveal new effects generated by this
extreme sparsity which are absent in \cite{kesten1975limit} and
\cite{BurDysIksMarRoi:2018+,matzavinos2016random}. While the
former article treats the nonsparse case in which random drifts
are imposed at each site, the latter two papers are concerned with
a sparse situation like here, the only difference being that
enormous gaps between the marked sites are prohibited. Now we
define the model in focus more precisely, as a particular random
walk in a random environment (RWRE). Also, we discuss its relation
to a multi-skewed Brownian motion and a one-dimensional trap
model.

We first recall the definition of a general RWRE.
Let $\Omega=(0,1)^\Z$ be the set of all possible environments
equipped with the corresponding Borel $\sigma$-algebra $\F$ and a
probability measure $P$. A random element
$\omega=(\omega_n)_{n\in\Z}$ defined on $(\Omega, \F, P)$ is
called {\it random environment}. A random walk in a random
environment $\omega$ is a nearest neighbor random walk $X =
(X_n)_{n \in \N_0}$ (here and hereafter, $\N_0 = \N \cup \{ 0 \}$)
on $\Z$. To define its transition probabilities we first set
$\mathcal{X} = \Z^{\N_0}$ and equip it with a Borel
$\sigma$-algebra $\mathcal{G}$. Plainly, $\mathcal{X}$ can be
thought of as the set of trajectories of $X$. Now, given $\omega$,
let $\P_\omega$ be a~ (quenched) probability measure on
$\mathcal{X}$ such that $X_0=0$ $\P_\omega$- almost surely (a.s.)
and
$$\P_\omega \{X_{n+1}=j| X_n = i\} = \left\{
\begin{array}{cl}
\omega_i, & \mbox{if } j=i+1,\\
1- \omega_i,\ & \mbox{if } j=i-1,\\
0, & \mbox{otherwise.}
\end{array}\right.$$
Clearly, under $\P_\omega$, $X$ is a time-homogeneous Markov chain
on $\Z$. The randomness of the environment $\omega$ influences
significantly various properties of $X$. In view of this, it is
quite natural to investigate the behavior of $X$ under the
annealed measure $\P$ which is defined as a unique probability
measure on $(\Omega\times \X, \F\otimes \G)$ satisfying
$$\P\{F\times G\} = \int_F \P_{\omega}\{G\} P({\rm d}\omega),\quad F\in \F,\quad G\in \G.$$

After these preparations we are ready to define the object of our
interest in the present paper. Denote by
$((\xi_k,\lambda_k))_{k\in \Z}$ a sequence of independent copies
of a random vector $(\xi,\lambda)$, where $\lambda \in (0,1)$ and
$\xi \in \N$ $\P$-a.s. Setting
$$S_n = \left\{ \begin{array}{cl}
\sum_{k=1}^n \xi_k, & \mbox{if } n>0, \\
0, & \mbox{if } n=0, \\
-\sum_{k=n+1}^{ 0} \xi_k, & \mbox{if } n<0.
\end{array} \right.$$
we define a specific random environment $\omega = (\omega_n)_{n \in \Z} \in \Omega$ by
\begin{equation}\label{eq:sparse}
\omega_n = \left\{
\begin{array}{ll}
\l_{k+1}, \ & \mbox{if $n=S_k$ for some }k\in\Z,\\
1/2, \ & \mbox{otherwise.}
\end{array}\right.
\end{equation}
Thus, the sequence $(S_n)_{n \in \Z}$ determines the marked sites in which the random drifts $\lambda_{k+1}$ are placed. Since for the nonmarked sites $n$ (that is, for most of sites) $\omega_n=1/2$, it is natural to call $\omega$ a {\it sparse random environment}. Following \cite{matzavinos2016random} we use the term {\it random walk in sparse random environment} (RWSRE) for $X$ as defined above with $\omega$ being a sparse random environment. We call the environment $\omega$ {\it moderately sparse} or {\it strongly sparse} depending on whether $\E\xi<\infty$ or
\begin{equation}\label{eq:infi}
\E\xi=\infty.
\end{equation}
While the case of moderate sparsity was analyzed in the recent
article \cite{BurDysIksMarRoi:2018+}, the case of strong sparsity
is investigated in the present work. In particular,
\eqref{eq:infi} is our main standing assumption.

The behavior of any RWRE is affected by both randomness of the
environment and randomness of the walk given the environment. The
earlier works \cite{BurDysIksMarRoi:2018+, kesten1975limit,
matzavinos2016random} demonstrate that in the nonsparse case
randomness of the environment has dominating effect. A remarkable
feature of the sparse case is that randomness of the environment
and randomness of the walk may contribute to a comparable extent.
Another source of the new effects arising in the sparse case are
the properties of the environment alone which are essentially
different from those in the nonsparse case.

Since the early work \cite{solomon1975random} a general RWRE has
been attracting a~fair amount of attention among probabilistic
community. We refer to \cite{zeitouni2004} for a classical
introduction to the topic. In the literature one usually treats
the cases where the environment $\omega$ forms a~stationary
ergodic sequence or just a collection of independent identically
distributed (iid) random variables (note that in the present
article we go beyond these settings). There are numerous articles
which prove, under these assumptions, quenched and annealed
distributional limit
theorems~\cite{bouchet2016quenched,dolgopyat2012quenched,enriquez2009limit,kesten1986limit,kesten1975limit,sinai1982limit,sznitman1999law}
and investigate large
deviations~\cite{buraczewski2017precise,comets2000quenched,dembo1996tail,
gantert1998quenched,greven1994large,pisztora1999large,pisztora1999precise,varadhan2003large,zerner1998lyapounov}.
The list above is far from being complete.

Further, we discuss a relation of RWSRE to two
other models. Following \cite{ramirez2011multi} we recall that a
multi-skewed Brownian motion evolves like a standard Brownian
motion with the exception of some deterministic sites (interfaces)
$(s_k)_{k \in \ZZ}$ in which some additional skewness
(perturbation) is imposed. Thus, the RWSRE can be seen as a
discrete analogue of a multi-skewed Brownian  motion with random
interfaces and random perturbations. To demonstrate a connection
to the other model we only observe the walker at the marked sites.
To be more precise, set
\begin{equation}\label{eq:trap}
\widehat{X}_0:=0,\quad \widehat{X}_n:=
\begin{cases}
\widehat{X}_{n-1},& \text{if } X_n\notin \{S_k:k\in\Z\},\\
k,& \text{if } X_n=S_k \text{ for some }k\in\Z;
\end{cases}
\quad n\in\N.
\end{equation}
and note that $\widehat{X}_n$ is the index of the last marked site
visited by the walker up to time $n$. The sequence $\widehat{X}=
(\widehat{X}_k)_{k\in\N_0}$ is a nearest-neighbor random walk on
$\Z$ in a random environment which has a positive probability to
stay immobile at any time. Thus, $\widehat{X}$ is a discrete
variant of a one-dimensional trap model
~\cite{Arous+Cherny:2005,Fontes+Isopi+Newman:2002,Zindy:2009}. The
setting of \cite{Zindy:2009} is closely related to that of the
present article. To justify the claim, assume that the
distribution of $\xi$ is heavy-tailed in the sense that
$\P\{\xi>t\}\sim t^{-\beta}$ as $t\to\infty$ for some $\beta\in
(0,1)$. Using the solution to gambler's ruin problem enables us to
calculate the transition probabilities of $\widehat X$ explicitly
and then conclude that $\P\{\tau>t\}$ is proportional to
$\P\{\xi^2>t\}$, where $\tau$ is the time of the first jump from a
given state (trapping time). One-dimensional trap models with
heavy-tailed trapping times are analyzed in \cite{Zindy:2009}. In
particular, it is shown that the corresponding nearest-neighbor
continuous-time Markov process, properly scaled, converges weakly
in the Skorokhod space to an inverse $\beta$- stable subordinator.
On the other hand, let us note right away that the assertions of
the present article cannot be derived from those in
\cite{Zindy:2009}. The explanation is simple: the evolution of $X$
between the marked sites is extremely important; thus, restricting
attention to the marked sites only leads to an essential loss of
information.

The article is organized as follows. In Section~\ref{sec:main} we
present our main results and review some earlier results which are
particularly relevant to ours. In Section~\ref{sec:branching} we
recall the construction of a branching process associated with
$X$. In Section \ref{sec:strategy} we explain our proof strategy.
In Section~\ref{sec:tails} several important auxiliary facts are
established. Finally, Section \ref{sec:proof} contains the proofs
of our main results.

\section{Main results}\label{sec:main}

\subsection{Preliminaries}\label{prelim}

In the paper \cite{matzavinos2016random} the authors address the
question of transience and recurrence of RWSRE and prove a strong
law of large numbers and some distributional limit theorems for
$X$. Put
\begin{equation*}
\rho = \frac{1-\lambda}{\lambda}.
\end{equation*}
According to Theorem 3.1 in \cite{matzavinos2016random}, $X$ is
$\P$-a.s.~transient to $+\infty$ whenever
\begin{equation}\label{eq:right_transience}
\E \log \rho \in [-\infty, 0)
\quad \text{and}\quad \E\log \xi<\infty
\end{equation}
(note that the first inequality in \eqref{eq:right_transience} excludes the degenerate case $\rho = 1$ a.s.\ in which $X$ becomes a simple random walk).
Under \eqref{eq:right_transience}, the RWSRE also satisfies a strong law of large numbers, that is,
\begin{equation}\label{eq:LLN}
\frac{X_n}{n} \to v \quad \PP-a.s.
\end{equation}
where
\begin{equation*}
v = \left\{ \begin{array}{cl}   \frac{(1 - \E \rho) \E \xi}{ (1-\E \rho) \E\xi^2 + 2 \E \rho \xi\E \xi} & \mbox{if  $\E \rho <1$,
$\E \rho\xi<\infty$ and $\E \xi^2<\infty$} \\
0 & \mbox{otherwise}     \end{array} \right. ,
\end{equation*}
see Theorem 3.3  in \cite{matzavinos2016random} and Proposition 2.1 in~\cite{BurDysIksMarRoi:2018+}.

We note right away that conditions \eqref{eq:infi}
and \eqref{eq:right_transience} are satisfied under the conditions
of our main results. Thus, the random walks in a sparse random
environment that we treat here are transient to the right with
zero asymptotic speed ($v=0$).

\subsection{Notation}

To state our main results we need more notation. Set
\begin{equation*}
T_n =\inf\{k\geq 0: X_k = n \},\quad n\in\mz.
\end{equation*}
We are going to derive limit theorems for $X_n$ from those for
$T_n$ via a standard inversion technique. It will become clear in
Section~\ref{sec:branching} that the stopping times $T_n$ are
easier to deal with, for, unlike $X_n$, these can be analyzed with
the help of an auxiliary branching process.

Now we formulate an assumption concerning the distribution of $\xi$:
\begin{itemize}\itemsep3mm
\item[($\xi$)] there exists $\beta\in (0,1]$ and a slowly varying function $\ell$ such that
\begin{equation}\label{eq:reg_var_xi}
\P\{ \xi > t \} \sim t^{-\beta} \ell (t) , \quad t \to \infty.
\end{equation}
and $\E \xi=\infty$ when $\beta=1$ ($\E\xi=\infty$ holds automatically when $\beta\in (0,1)$).
\end{itemize}
Further, we point out two sets of assumptions regarding the distribution of $\rho$:
\begin{itemize}\itemsep3mm
        \item[($\rho 1$)] $\E \rho^{\alpha}=1$ for some
        $\alpha>0$, $\E \rho^{\gamma}<\infty$ for some $\gamma>\alpha$ and the distribution $\log \rho$ is nonarithmetic;
        \item[($\rho 2$)] there exists an open interval $\mathcal{I}\subset (0,\infty)$ such that $\E \rho^x<1$ for all $x\in \mathcal{I}$.
\end{itemize}
Note that ($\rho 1$) and ($\rho 2$) are not disjoint because
($\rho 1$) implies ($\rho 2$) with $\mathcal{I}\subset
(0,\alpha)$.

To ease the presentation we shall state separately our results for $\beta\in (0,1)$ and $\beta =1$ because the latter case is technically more involved.

\subsection{Results for $\beta\in (0,1)$}\label{beta<1}

We shall need two assumptions concerning the joint distribution of $(\xi,\rho)$:
\begin{itemize}\itemsep3mm
\item[($\xi\rho 1$)] $\lim_{t\to\infty}\frac{\P\{\xi>t^{1/2},\, \rho>c_1(t)\}}{\P\{\xi>t\}}=0$,
where $c_1(t)=t$;

\item[($\xi\rho 2$)] $\lim_{t\to\infty}\frac{\P\{\xi>t^{\alpha/\beta},\, \rho>c_2(t)\}}{\P\{\xi>t\}}=0$, where
$c_2(t):=t^{-1}\P\{\xi>t\}^{-1/\alpha}$. For the
most part, $\alpha$ is supposed to be the same as in ($\rho 1$).
But occasionally we allow $\alpha$ to be any positive number
satisfying $\alpha\leq \beta/2$.
\end{itemize}
An application of Markov's inequality reveals that under $(\rho1)$
with $\alpha=\beta/2$ or $(\rho2)$ with $\beta/2\in\mathcal{I}$
condition ($\xi\rho1$) holds whenever $\xi$ and $\rho$ are
independent. Similarly, under $(\rho1)$ with $\alpha\in
(0,\beta/2]$ condition ($\xi\rho2$) holds provided that $\xi$ and
$\rho$ are independent. Further, it is clear that, for $i=1,2$,
$\P\{\rho>c_i(t)\} = o(\P\{\xi>t\})$ is a sufficient condition for
($\xi\rho\,{\rm i}$) which is far from being necessary.

Denote by $\vartheta$ a positive random variable with Laplace
transform
\begin{equation}\label{varthe}
\E\exp(-s\vartheta)=\frac{1}{\cosh\sqrt{s}},\quad s\geq 0.
\end{equation}
Define the measure $\mu$ on $\mathbb{K}:=[0,\infty]^2\setminus  \{
(0,0) \}$ by
\begin{equation}\label{mu}
\mu\{(u,v)\in \mathbb{K}: u>x_1~\text{or}~v>x_2\}=x_1^{-\beta}+\mathcal{C}_\mu x_2^{-\beta/2}-\E\min(x_1^{-\beta}, x_2^{-\beta/2}\vartheta^{\beta/2})
\end{equation}
for $ x_1, x_2>0$, where $\mathcal{C}_\mu>0$ is a constant to be
specified later. Let $N:= \sum_k \delta_{(t_k,\, {\bf j}_k)}$ be a
Poisson random measure on $[0,\infty)\times \mathbb{K}$ with
intensity measure ${\rm LEB}\otimes \mu$. Here, $\delta_{(t,{\bf
x})}$ is the probability measure concentrated at $(t, {\bf x})\in
[0,\infty)\times \mathbb{K}$, and ${\rm LEB}$ is the Lebesgue
measure on $[0,\infty)$. Set
\begin{equation}\label{lt}
{\bf L}(t):=(L_1(t), L_2(t))=\sum_{ k \: : \: t_k\leq t}{\bf
j}_k,\quad t\geq 0.
\end{equation}
Lemma~\ref{levymeasure} given in Section~\ref{sec:proof} secures
\begin{equation}\label{ineq}
\int_{|x|\neq 0}(|{\bf x}|\wedge 1) \: \mu({\rm d}{\bf x})<\infty,
\end{equation}
where $|{\bf x}|=\sqrt{x_1^2+x_2^2}$ for $x=(x_1, x_2)\in \R^2$.
This ensures that the series on the right-hand side of \eqref{lt}
converges a.s. Furthermore, $({\bf L}(t))_{t\geq 0}$ is a
two-dimensional (non-stable) L\'{e}vy process with the L\'{e}vy
measure $\mu$. Its components $L_1$ and $L_2$ are dependent
drift-free stable subordinators with parameters $\beta$ and
$\beta/2$, respectively. Put
$$L_1^\leftarrow(t):=\inf\{s\geq 0: L_1(s)>t\},\quad t\geq 0.$$ The process $(L_1^\leftarrow(t))_{t\geq 0}$ is known in the literature as the inverse
$\beta$-stable subordinator. Set
$$\chi:=L_2(L_1^\leftarrow(1)-)+\vartheta
(1-L_1(L_1^\leftarrow(1)-))^2,$$ where $\vartheta$ is assumed to
be independent of ${\bf L}$.

Our first result, Theorem \ref{thm:main11T}, provides a distributional limit theorem for $X$ in the situation where the distribution of $\xi$ plays a dominant role, a contribution of the distribution of $\rho$ being small.
\begin{thm}\label{thm:main11T}
Let $\beta\in (0,1)$ and assume that one of the
following sets of  conditions is satisfied:
\begin{itemize}
\item (A) ($\xi$) and ($\rho 2$) with $\beta/2 \in {\mathcal I}$ hold;
\item (B1) ($\xi$) holds with $\ell$ such that $\lim_{t\to\infty}\ell (t)=\infty$, and ($\rho 1$) holds with $\alpha = \beta/2$;
\item (B2) ($\xi$) holds with $\ell$ such that $\mathcal{C}_\ell:=\lim_{t\to\infty}\ell (t)\in (0,\infty)$,
and ($\rho 1$) holds with $\alpha = \beta/2$.
\end{itemize}
Further, assume that condition $(\xi\rho 1)$ holds and that
$\E(\rho\xi)^{\beta/2+\varepsilon}<\infty$ for some
$\varepsilon>0$. Then
\begin{equation}\label{Tn}
\frac{T_n}{n^2}~\dod~2\chi,\quad n\to\infty
\end{equation}
and
\begin{equation}\label{Xn}
\frac{X_n}{n^{1/2}}~\dod~(2\chi)^{-1/2},\quad n\to\infty.
\end{equation}
The constant $\mathcal{C}_\mu$ in \eqref{mu} is given as follows:
\begin{equation*}
\mathcal{C}_\mu = \left\{ \begin{array}{cl} \E\vartheta^{\beta/2} & \mbox{in the cases (A) and (B1)} \\
\mathcal{C}_\ell \E\vartheta^{\beta/2}+\mathcal{C}_Z(\beta/2) &
\mbox{in the case (B2)}
\end{array} \right. ,
\end{equation*}
where the constant $\mathcal{C}_Z(\beta/2)$ is
specified in Lemma~\ref{prop:main2}.
\end{thm}

In Theorem \ref{thm:main11T} the condition
$\E\rho^\alpha=1$ may hold for any $\alpha>0$. In the situations
(B1) and (B2) (as well as in Theorem \ref{thm:main2T} given below)
it holds for $\alpha\in (0,\beta/2]$. Assuming that it holds for
$\alpha>\beta/2$ we conclude that ($\rho 2$) with $\beta/2\in
\mathcal{I}$ holds, so that the situation (A) prevails.

Here is a very informal explanation of why $n^2$
should be the correct normalization for $T_n$. Since the
distribution of $\xi$ dominates that of $\rho$, it is tempting to
assume, at least as a first approximation, that $\rho=1$ $\P$-a.s.
Then $X$ is a SRW, and the fact that $T_n/n^2$ converges in
distribution is well-known.

When $\xi$ and $\rho$ are independent, the conditions $(\xi\rho
1)$ and $\E(\rho\xi)^{\beta/2+\varepsilon}<\infty$ for some
$\varepsilon>0$ are secured by the other assumptions of Theorem
\ref{thm:main11T} (as for $(\xi\rho 1)$, see the discussion at the
beginning of Section \ref{beta<1}).

Theorem \ref{thm:main2T} given next is a counterpart of Theorem
\ref{thm:main11T} in which the distributions of $\xi$ and $\rho$
play comparable roles. To formulate it we need some additional
notation. Pick any $\alpha\in (0,\beta/2]$ and denote by
$\widehat{L}_1:=(\widehat{L}_1(t))_{t\geq 0}$ and
$(\widehat{L}_2(t))_{t\geq 0}$ {\it independent} drift-free
$\beta$- and $\alpha$-stable subordinators with the L\'{e}vy
measures $\nu_1$ and $\nu_2$ given by
$$\nu_1((x,\infty))=x^{-\beta},\quad \nu_2((x,\infty))=\mathcal{C}_Z(\alpha) x^{-\alpha}, \quad
x>0,$$ respectively (see Lemma \ref{prop:main2} for the definition
of $\mathcal{C}_Z(\alpha)$). Also, let
$(\widehat{L}_1^\leftarrow(t))_{t\geq 0}$ denote an inverse
$\beta$-stable subordinator which corresponds to $\widehat{L}_1$.
Finally, whenever $(\xi)$ holds we denote by $\lambda$ an
asymptotic inverse function for $s\mapsto
\P\{\xi>s\}^{-1/\alpha}$. This means that $\lambda$ satisfies
\begin{equation*} \lim_{t \to \infty} \frac{\PP\{\xi >
\lambda(t)\}^{-1/\alpha}}{t} = \lim_{t \to \infty} \frac{\lambda
(\PP\{ \xi > t \}^{-1/\alpha})}{t}=1.
\end{equation*}
Such a function $\lambda$ is uniquely determined up to asymptotic
equivalence by Theorem 1.5.12 in \cite{bingham1989regular}.
Moreover, it is regularly varying of index $\alpha/\beta$.
\begin{thm}\label{thm:main2T}
Let $\beta\in (0,1)$ and assume that one of the
following sets of  conditions is satisfied:
\begin{itemize}
\item (B3) ($\xi$) holds with $\ell$ such that $\lim_{t\to\infty}\ell (t)=0$, and ($\rho 1$) holds with $\alpha = \beta/2$;
\item (C) ($\xi$) holds, and ($\rho 1$) holds with $\alpha\in (0, \beta/2)$.
\end{itemize}
Further, assume that condition $(\xi\rho2)$ holds and that
$\E(\rho\xi)^{\alpha+\varepsilon}<\infty$ for some $\varepsilon>0$. Then
\begin{equation}\label{Tn2}
\P\{\xi>n\}^{1/\alpha}
T_n~\dod~2\widehat{L}_2(\widehat{L}_1^\leftarrow(1)),\quad
n\to\infty
\end{equation}
and
\begin{equation}\label{Xn2}
\frac{X_n}{\lambda(n)}~\dod~(2\widehat{L}_2(\widehat{L}_1^\leftarrow(1)))^{-\alpha/\beta},\quad
n\to\infty.
\end{equation}
\end{thm}

An informal but well-justified explanation of why
the normalization in Theorem \ref{thm:main2T} is plausible
inevitably requires introducing a new notation that we prefer to
avoid at this stage. Thus, we only note, without going into
details, that the normalization $\P\{\xi>n\}^{-1/\alpha}$ which is
different from that in Theorem \ref{thm:main11T} is given by the
composition $(f\circ g)(n)$, where $f(x)=x^{1/\alpha}$ and
$g(x)=\P\{\xi>x\}^{-1}$. The functions $g$ and $f$ represent the
contributions of the distributions of $\xi$ and $\rho$,
respectively.

When $\xi$ and $\rho$ are independent, the
conditions $(\xi\rho 2)$ and
$\E(\rho\xi)^{\alpha+\varepsilon}<\infty$ for some $\varepsilon>0$
are implied by the other assumptions of Theorem \ref{thm:main2T}
(as for $(\xi\rho 2)$, see the discussion at the beginning of
Section \ref{beta<1}).

\subsection{Results for $\beta=1$}\label{beta1}

The boundary case $\beta=1$ is essentially simpler
but technically more involved than the case $\beta\in (0,1)$.

Whenever \eqref{eq:reg_var_xi} holds (for $\beta\in (0,1]$) we
denote by $a$ any positive measurable function satisfying
\begin{equation}\label{eq:k50}
\lim_{t\to\infty}t\P\{\xi>a(t)\}=1.
\end{equation}
Further, we put, for $t>0$,
\begin{equation}\label{eq: mt}
m(t)=\int_0^t\P\{\xi>u\}{\rm d}u,\quad\mbox{ and } \quad
\pi(t):=m(a(t))
\end{equation}
and define a positive measurable function $\pi^{\ast}$ such that
$$\lim_{t\to\infty}\pi(t)\pi^{\ast}(t\pi(t))=\lim_{t\to\infty}\pi^{\ast}(t)\pi(t\pi^{\ast}(t))=1.$$
Since $\beta=1$ and $\E\xi=\infty$, $m$ and $\pi$ are slowly
varying and unbounded, and $\pi^\ast$ is a de~Bruijn conjugate
function for $\pi$, see Theorem 1.5.13 in
\cite{bingham1989regular}. In the present case $\beta=1$ we shall
use conditions ($\xi\rho 1$) and ($\xi\rho 2$) introduced in
Section \ref{beta<1} with
\begin{equation}\label{eq: cc}
c_1(t):=a(t\pi^\ast(t)), \qquad \mbox{ and } \qquad
c_2(t):=t^{-1}(t\pi^{\ast}(t))^{1/\alpha}.
\end{equation}
These functions may be well-defined for large $t$
only which is sufficient for our purposes. Although we are not
going to use this observation, let us note that the so defined
$c_i$ are asymptotically equivalent, up to multiplicative
constants, to the $c_i$ defined in the conditions ($\xi\rho\,{\rm
i}$) in Section \ref{beta<1} for the case $\beta\in (0,1)$.

When $\beta=1$, the two-dimensional subordinator ${\bf L}$ defined
in \eqref{lt} does not exist simply because \eqref{ineq} does not
hold any longer. However, its second component $L_2$ is still
well-defined. Actually, it is a drift-free stable subordinator
with parameter $1/2$ and the L\'{e}vy measure $\mu_2$ given by
$\mu_2((x,\infty))=\mathcal{C}_\mu x^{-1/2}$ for $x>0$. As a final
preparation, denote by $w$ and $\kappa$ asymptotic inverse
functions for $s\mapsto a(s\pi^\ast(s))^2$ and $s\mapsto
(s\pi^{\ast}(s))^{1/\alpha}$, respectively, where $\alpha>0$.
Since $s\mapsto a(s\pi^\ast (s))^2$ and $s\mapsto
(s\pi^{\ast}(s))^{1/\alpha}$ are regularly varying of indices $2$
and $1/\alpha$, such $w$ and $\kappa$ are regularly varying
functions of indices $1/2$ and $\alpha$.
\begin{thm}\label{thm:main3T}
Assume that the assumptions of Theorem \ref{thm:main11T} are
satisfied for $\beta=1$. Then
\begin{equation}\label{Tn_beta_1}
\frac{T_n}{a(n\pi^\ast (n))^2}~\dod~2L_2(1),\quad n\to\infty,
\end{equation}
and
\begin{equation}\label{Xn_beta_1}
\frac{X_n}{w(n)}~\dod~(2L_2(1))^{-1/2},\quad n\to\infty.
\end{equation}
\end{thm}

\begin{thm}\label{thm:main3T1}
Assume that the assumptions of Theorem \ref{thm:main2T} are
satisfied for $\beta=1$. Then
\begin{equation}\label{Tn2_beta_1}
\frac{T_n}{(n\pi^{\ast}(n))^{1/\alpha}} ~\dod~2 L_2(1), \quad
n\to\infty
\end{equation}
and
\begin{equation*}
\frac{X_n}{\kappa (n)}~\dod~(2L_2(1))^{-\alpha}, \quad n\to\infty.
\end{equation*}
\end{thm}

\subsection{Comparison to earlier limit theorems}

It is more convenient to discuss limit results for $T_n$ rather
than $X_n$. Distributional limit theorems for $X_n$ and $T_n$ are
proved in~\cite{matzavinos2016random} for the case where $\xi$ is
$\P$-a.s.\ bounded (the corresponding environment may be called
{\it weakly sparse}). Then, as expected, the distribution of $\xi$
does not affect the asymptotic behavior of $T_n$ in a significant
way. The key parameter is $\alpha>0$ for which $\E \rho^\alpha =1$
(as in $(\rho 1)$), and $T_n$, properly normalized and centered,
converges in distribution to an $\alpha$-stable law (if $\alpha
\geq 2$ the corresponding limit is Gaussian), see Proposition
3.9~\cite{matzavinos2016random}. For instance, if $\alpha \in
(0,1)$, then `$T_n$ grows like $n^{1/\alpha}$'. The~arguments rely
on a change of measure which transfers the RWSRE into a random
walk in a Markov environment. As a consequence of $\P$-a.s.\
boundedness of $\xi$, the Markov chain driving the environment has
a~finite state space which makes certain results of
\cite{mayer2004limit} applicable.

In order to go beyond bounded $\xi$ one has to develop a different
approach (one possibility exploited both in
\cite{matzavinos2016random} and in the present article is to use a
link with certain branching processes with immigration in a random
environment). In the case of moderately sparse environment, that
is, $\E \xi <\infty$ it is shown in~\cite{BurDysIksMarRoi:2018+}
that the asymptotics $T_n$ strongly depends on the interplay of
parameters $\alpha$ and $\beta$ and the behavior of a slowly
varying function $\ell$ (see conditions $(\rho 1)$ and $(\xi)$ for
the definition). In all cases, the limit distribution of $T_n$,
properly normalized and centered, is still stable. However, the
normalization is $n^{1/\alpha}$ when the distribution of $\rho$
dominates that of $\xi$, whereas it is $n^{1/\alpha}L(n)$ for a
slowly varying $L$ when the distribution of $\xi$ dominates that
of $\rho$. Summarizing, the results of both
\cite{matzavinos2016random} and \cite{BurDysIksMarRoi:2018+} bear
a strong resemblance with those of \cite{kesten1975limit} which is
concerned with the nonsparse case $\xi=1$ $\P$-a.s.

On the technical level the difference between the
cases of moderate and strong sparsity is carefully explained at
the beginning of Section \ref{sec:strategy}. The strong sparsity
strongly manifests itself in Theorem \ref{thm:main11T}. In it,
unlike the earlier limit theorems discussed above the
normalization for $T_n$ is $n^2$ as if $(X_n)$ was a SRW. However,
the connection with a SRW does not extend to the limit
distribution which is rather exotic and seems to be new in the
context of RWRE and in general. In Theorem \ref{thm:main2T} the
limit distribution is still nonstable. However, the limit result
obtained here looks more similar to those in the moderately sparse
case. Loosely speaking, the normalization in the cited theorem can
be interpreted as the time-changed version of $n^{1/\alpha}$. The
case $\beta=1$ treated in Theorems \ref{thm:main3T} and
\ref{thm:main3T1} can be thought of as {\it almost moderate}. The
closeness to the moderate sparsity only reflects in a stable limit
distribution, whereas the normalization for $T_n$ is different
from those appearing in \cite{BurDysIksMarRoi:2018+}.

\section{An associated branching process}\label{sec:branching}

\subsection{The construction}\label{sec:br1}

The relation between certain random walks and
branching processes goes back to Harris~\cite{harris:1952}. Later
on, it was successfully applied, in an extended form, in
\cite{kesten1975limit} to obtain distributional limit theorems for
random walks in an iid random environment. Since then branching
processes have become a useful tool in the analysis of
one-dimensional RWRE. The presentation below follows closely that
in \cite{kesten1975limit} or~\cite{BurDysIksMarRoi:2018+}.

Fix $n \in \NN$ and consider the random variables $$U_i^{(n)} = \#
\big\{ k < T_n: X_k=i, X_{k+1} = i-1 \big\}, \qquad i\le n.$$
Since $X_{T_n}=n$ and $X_0=0$ we have, for $n\in \N$,
\begin{eqnarray*}
T_n &=& \# \mbox{ of steps during $[0,T_n)$}\\
&=& \# \mbox{ of steps to the right during $[0,T_n)$}+ \# \mbox{ of steps to the left during $[0,T_n)$}\\
&=& n + 2 \cdot \# \mbox{ of steps to the left during $[0,T_n)$}
\end{eqnarray*}
which gives
\begin{equation}\label{basic}
T_n = n + 2\sum_{i=-\infty}^n U_i^{(n)},\quad n\in\N.
\end{equation}
Recall from Section \ref{prelim} that, under the
setting of the present paper, $X$ is transient to the right, that
is, $\lim_{n\to\infty}X_n=+\infty$ $\P$-a.s. This entails
\begin{equation}\label{eq:negative_half_line_negligible}
\sum_{i < 0}U_i^{(n)}\leq \text{ total time spent by } X \text{ in
}(-\infty,0) < \infty\quad \P-\text{a.s.}
\end{equation}
In particular,
\begin{equation}\label{eq:t_n_decomposition1}
T_n = n + 2\sum_{i=0}^n U_i^{(n)} + O_\P(1),\quad
n\in\N,
\end{equation}
where $O_\P(1)$ is a term which is bounded in probability. As a
consequence, distributional limit theorems for $T_n$ will follow
from those for $n+2\sum_{i=0}^n U_i^{(n)}$. The latter variables
possess an elegant stochastic structure since, as argued below,
for fixed environment $\omega$, $U_n^{(n)}, U_{n-1}^{(n)},
U_{n-2}^{(n)}, \ldots U_0^{(n)}$ form a sequence of the first $n$
generations of a inhomogeneous branching process with unit
immigration.

In what follows, for $p\in (0,1)$, ${\rm Geom}(p)$ is a shorthand
for a geometric distribution with success probability $p$, that
is, $${\rm Geom}(p)\{\ell\}=p(1-p)^\ell,\quad \ell\in\N_0.$$ Note
that $U_n^{(n)}=0$ and that $U_{n-1}^{(n)}$ is equal to the number
of excursions to the left of $n-1$ before the first visit to $n$.
Due to the transitivity of $X$, $U_{n-1}^{(n)}$ is distributed
according to ${\rm Geom}(\omega_{n-1})$. Further, observe that,
for $i=1,\ldots, n-2$, $U_{n-i-1}^{(n)}$ can be represented as
follows:
\begin{equation}\label{eq: gr1}
U_{n-i-1}^{(n)} = \sum_{k=1}^{U_{n-i}^{(n)}} V_k^{(n-i-1)} +
V_0^{(n-i-1)},
\end{equation}
where, for $k\in\N$, $V_k^{(n-i-1)}$ denotes the number of
excursions to the left of $n-i-1$ during the $k$th excursion to
the left of $n-i$ and $V_0^{(n-i-1)}$ is the number of excursions
to the left of $n-i-1$ before the first excursion to the left of
$n-i$. Notice, since $X$ is transitive and enjoys the Markov
property with respect to the quenched probability, for each fixed
$\omega$, $V_1^{(n-i-1)}$, $V_2^{(n-i-1)},\ldots$ are independent
random variables with distribution ${\rm Geom}(\omega_{n-i-1})$
which are also independent of $U_{n-i}^{(n)}$. This shows that,
under $\P_\omega$, $U_n^{(n)}, U_{n-1}^{(n)}, U_{n-2}^{(n)},
\ldots, U_0^{(n)}$ are the consecutive generation sizes in a
inhomogeneous branching process with unit immigration in which the
particles and the immigrant in the $(i-1)$th generation
($i=1,\ldots, n-1$) reproduce according to ${\rm
Geom}(\omega_{n-i})$ distribution.

To ease the notation, we introduce another branching process $Z =
(Z_k)_{k\geq 0}$ which evolution can be described as follows. We
start with $Z_0=0$ particles. At the generation $n=1$, the first
immigrant enters the system and gives birth to $Z_1=G_0^{(1)}$ new
particles, where $G_0^{(1)}$ has distribution  ${\rm
Geom}(\omega_1)$. At the generation $n=2$, another immigrant
enters the system and all $Z_1+1$ particles reproduce
independently according to distribution ${\rm Geom}(\omega_2)$.
The offspring of the first generation particles (including the
immigrant) form the second generation. The~system evolves
according to these rules, with one new immigrant entering the
system at each generation. In general, for each $n\in\N$, $Z_n$
admits the following representation
\begin{equation}\label{eq:z}
Z_n = \sum_{k=1}^{Z_{n-1}} G_k^{(n)} + G_0^{(n)},
\end{equation}
where $G_0^{(n)}$ is the number of offspring of the $n$th
immigrant and $G_k^{(n)}$ is the number of offspring of the $k$th
particle in the $(n-1)$st generation (we set $G_k^{(n)}=0$ if the
$k$th particle in the $(n-1)$st generation does not exist). Thus,
the process $Z$ does not count the immigrants. Plainly, under
$\P_\omega$, for each $n\in\N$, $G_0^{(n)}$, $G_1^{(n)},\ldots$
are independent random variables with distribution ${\rm
Geom}(\omega_n)$ which are independent of $Z_n$. Whenever the
environment is sparse, while the process $Z$ reproduces according
to distribution ${\rm Geom}(\lambda_{k+1})$ at time $S_k$ for
$k\in\N_0$, most of the time, it evolves as a critical
Galton--Watson process with unit immigration and the offspring
distribution ${\rm Geom}(1/2)$, to be denoted by $\Zc = (\Zc_n)_{n
\in \NN_0}$. In particular, for $n\in\N$, given $(\xi_j,
\rho_j)_{1\leq j\leq n}$,
\begin{equation}\label{eq:cond}
\sum_{i=0}^{S_n} U_i^{(S_n)} \stackrel{{\rm d}}{=}
\sum_{k=0}^{S_n}Z_k.
\end{equation}

For later needs we note that $\Zc$ is a particular
instance of the process $Z$ which corresponds to $\omega_n=1/2$.
In particular, $\Zc$ satisfies \eqref{eq:z} in which
$(G_k^{(n)})_{k\in \N_0, n\in\N}$ are independent random variables
with distribution ${\rm Geom}(1/2)$.

\subsection{The notation}\label{subsec:notation}

For $k,n\in\N$, denote by $Z(k,n)$ the number of
progeny residing in the $n$th generation of the $k$th immigrant.
In particular, $Z(k,k)$ is the number of offspring of this
immigrant and $$Z_n = \sum_{k=1}^n Z(k,n),\quad n\in\N.$$
Moreover, for each $k \in \NN$, $(Z(k, n))_{n \geq k}$ forms a
branching process in a random environment (without immigration).
Since at each generation the random reproduction law is the same
for all particles, the processes $(Z(1, n))_{n \geq 1}$, $(Z(2,
n))_{n \geq 2}$, $(Z(3, n))_{n \geq 3}$ \ldots are dependent with
respect to the annealed probability but are independent with
respect to the quenched probability. For $n\in\N$ and $1\leq i\leq
n$, let $Y(i,n)$ denote the number of progeny in the generations
$i,i+1,\ldots, n$ of the $i$th immigrant, that is,
$$Y(i,n) = \sum_{k=i}^n Z(i,k).$$ Similarly, for $i\in\N$, we
denote by $Y_i$ the total progeny of the $i$th immigrant, that is,
$$Y_i=Y(i,\infty)=\sum_{k\geq i}Z(i,k).$$ We also define $W_n$ to
be the total population size in the first $n$ generations, that
is, $$W_n=\sum_{j=1}^n Z_j,\quad n\in\mathbb{N}.$$ In view of the
structure of the environment it is natural to divide the
population into blocks which include generations $1,\ldots, S_1$;
$S_1+1,\ldots, S_2$ and so on. To set out the necessary notation,
we write $$\Z_n = Z_{S_n},\quad n\in\N$$ for the number of
particles in the generation $S_n$, and $$\W_n = W_{S_n} -
W_{S_{n-1}}= \sum_{j=S_{n-1}+1}^{S_n}Z_j,\quad n\in\N$$ for the
total population in the generations $S_{n-1}+1,\ldots, S_n$.

Put $\Wc_n=\sum_{k=1}^n \Zc_k$ for $n\in\N$, so
that $\Wc_n$ is the total progeny in the first $n$ generations of
$\Zc$, see the end of Section \ref{sec:br1} for the definition of
$\Zc$. It is known that
\begin{equation}\label{distr_conv}
n^{-2} \Wc_n~\dod~ \vartheta,\quad n\to\infty,
\end{equation}
where $\vartheta$ is a random variable with the Laplace transform
given in~\eqref{varthe}, see Theorem 5
in~\cite{pakes1972critical}, and that
\begin{equation}\label{eq:w57}
\lim_{n\to\infty}  \E(  n^{-2} \Wc_n)^s = \E \vartheta^s, \qquad s
>0,
\end{equation}
see Lemma 6.5 in \cite{BurDysIksMarRoi:2018+}. These properties of
$\Zc$ will play an essential role in our proofs.

\section{The strategy of the proofs}\label{sec:strategy}

To prove a distributional limit theorem for $T_n$
it is natural to use a decomposition
\begin{equation}\label{eq:dec}
T_n=T_{S_{\nu(n)-1}}+(T_n-T_{S_{\nu(n)-1}}),\quad n\in\N,
\end{equation}
where
\begin{equation*}
\nu(t) = \inf\{ n \in \N \: : \: S_n >t \},\quad t\geq 0.
\end{equation*}
In principle, the asymptotic behavior of $T_{S_{\nu(n)-1}}$ may be
regulated by that of $S_n$, $W_{S_n}$ or both, {see formulae \eqref{eq:t_n_decomposition1} and \eqref{eq:cond}}. In the paper
\cite{BurDysIksMarRoi:2018+} which treats the case of moderate
sparsity $\E\xi<\infty$ while the contribution of the second
summand is negligible, the asymptotics of the first summand
$T_{S_{\nu(n)-1}}$ is driven by $W_{S_n}$ alone. The latter is
explained by the fact that the contribution of $S_n$ is only seen
in the form of a law of large numbers and, as such, degenerate in
the limit. The case of strong sparsity $\E\xi=\infty$ we are
interested in here is more involved. Indeed, now the asymptotics
of $T_{S_{\nu(n)-1}}$ is affected by $(S_n, W_{S_n})$, for, under
($\xi$), $S_n$, properly normalized, converges in distribution to
a nondegenerate random variable. Further, in Theorem
\ref{thm:main11T} the contributions of the summands in
\eqref{eq:dec} are comparable. Therefore, one has to investigate
their joint asymptotic behavior which leads to technical
complications. On the other hand, the asymptotics of
$T_n-T_{S_{\nu(n)-1}}$ alone is relatively easy to deal with, for
the principal component of this random variable is given by the
first-passage time of a reflected SRW stopped at an independent
time. The other main results, Theorems \ref{thm:main2T},
\ref{thm:main3T} and \ref{thm:main3T1}, are simpler than Theorem
\ref{thm:main11T} because the first summand in \eqref{eq:dec}
dominates the second.

The text below is borrowed from Section 4 in
\cite{BurDysIksMarRoi:2018+}, with minor alterations and
additions. While dealing with $W_{S_n}$ our main arguments follow
the strategy invented by Kesten et al.~\cite{kesten1975limit}.
Namely, for large $n$, we decompose $W_{S_n}$ as a sum of random
variables which are iid under the annealed probability $\P$. For
this purpose we define extinction times
\begin{equation}\label{eq:tau_def}
\tau_0:=0,\quad \tau_{k} := \min \{j>\tau_{k-1} : \Z_j = 0\},\quad
k\in\mathbb{N}.
\end{equation}
Let us emphasize that the extinctions of $Z$ in the generations
other than $S_1$, $S_2,\ldots$ are ignored. Set
$$\overline{\W}_{\tau_n}:=W_{S_{\tau_n}} -
W_{S_{\tau_{n-1}}},\quad n\in\N$$ and note that
$(\overline{\W}_{\tau_n},\tau_n-\tau_{n-1})_{n\in\N}$ are iid
random vectors under $\PP$. Since the random variables in question
are non-negative we have, for $n\in\N$,
\begin{equation}\label{eq:main_two_sided_estimate}
\sum_{k=1}^{\tau^{\ast}_n}\overline{\W}_{\tau_k}\leq
\sum_{k=1}^{S_n} Z_k\leq \sum_{k=1}^{\tau^{\ast}_n+1}
\overline{\W}_{\tau_k}\quad \P-\text{a.s.,}
\end{equation}
where $\tau_n^\ast$ is the number of extinctions of $Z$ in the
generations $S_0,\ldots, S_n$, that is, $$\tau^{\ast}_n:=\max
\{k\geq 0:\tau_k\leq n\},\quad n\in\mathbb{N}.$$

Lemma \ref{lem:nu} given next states that the
extinctions occur rather often.
\begin{lem}\label{lem:nu}
Assume that $\E\log \rho\in [-\infty, 0)$ and $\E\log \xi<\infty$.
Then $\E\tau_1 <\infty$. If additionally $\E
\rho^\varepsilon<\infty$ and $\E \xi^\varepsilon<\infty$ for some
$\varepsilon>0$, then $\me \exp (\gamma \tau_1)<\infty$ for some
$\gamma>0$.
\end{lem}

The proof of this lemma can be found in the Appendix
of~\cite{BurDysIksMarRoi:2018+}. Under the assumptions of our main
results Lemma \ref{lem:nu} ensures that ${\tt m}=\E\tau_1<\infty$.
The strong law of large numbers for renewal processes
$(\tau_n^\ast)_{n \in \N_0}$ makes it plausible that, for large
$n$,
\begin{equation*}
W_{S_n}~ \approx~ \sum_{k=1}^{[{\tt m}^{-1}n]}
\overline{\W}_{\tau_k}.
\end{equation*}
The right-hand side, properly centered and normalized, converges
in distribution if, and only if, the distribution of
$\overline{\W}_{\tau_1}$ belongs to the domain of attraction of a
stable law. According to Lemma
\ref{prop:tail_main}, the latter is indeed the case under the
assumptions of our theorems.

An important technical ingredient of our proofs is the
distribution tail behavior of the vector $(S_{\tau_1},
\overline{\W}_{\tau_1})$. To investigate it we have to discuss the
structure of $\overline{\W}_{\tau_1}$ in more details. To this
end, for $i\in\N$, we divide particles residing in the generations
$S_{i-1}+1,\ldots, S_i$ into groups:
\begin{itemize}
\item $\mathcal{P}_{1,i}$~--~the progeny residing in the generations $S_{i-1}+1,\ldots, S_i-1$ of the immigrants arriving
in the generations $S_{i-1},\ldots, S_i-2$, the number of these
being $$\W^0_i:=\sum_{j=S_{i-1}+1}^{S_i-1}\sum_{k=j}^{S_i-1}
Z(j,k);$$

\item $\mathcal{P}_{2,i}$~--~ the progeny residing in the generations $S_{i-1}+1,\ldots, S_i-1$ of the immigrants arriving in the generations $0,1,\ldots, S_{i-1}-1$,
the number of these being $$\W^\downarrow_i:=
\sum_{j=1}^{S_{i-1}}\sum_{k=S_{i-1}+1}^{S_i-1} Z(j,k);$$

\item $\mathcal{P}_{3,i}$~--~ particles of the generation $S_i$, the number of these being $\Z_i$.
\end{itemize}

The aforementioned partition of the population which is depicted
on Figure \ref{fig:population} induces the following
decompositions which hold $\P$-a.s. $$\W_i =
\W^0_i + \W^\downarrow_i +\Z_i,\quad i\in\N$$ and
\begin{equation}\label{eq:wp3}
\overline{\W}_{\tau_1}=\sum_{i=1}^{\tau_1} \W^0_i +
\sum_{i=1}^{\tau_1} \W^\downarrow_i +\sum_{i=1}^{\tau_1}\Z_i.
\end{equation}

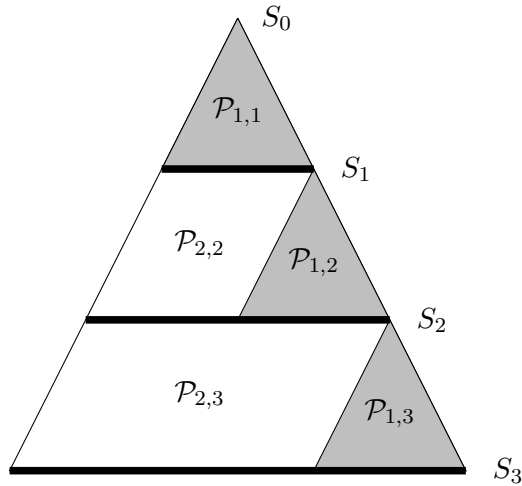
\begin{figure}[ht]
\centering
\begin{tikzpicture}
        \draw (0,6) node[right=5pt]{$S_0$}; \draw (0,6) -- (3,0) --
        (-3,0)-- (0,6); \draw[fill=gray!50] (0,6) -- (1,4) -- (-1,4) --
        (0,6) node[below=26pt]{$\mathcal{P}_{1,1}$}; \draw[fill=gray!50]
        (1,4) -- (2,2) -- (0,2) -- (1,4)
        node[below=26pt]{$\mathcal{P}_{1,2}$}; \draw[fill=gray!50] (2,2)
        -- (3,0) -- (1,0) -- (2,2) node[below=26pt]{$\mathcal{P}_{1,3}$};
        \draw[line width=1mm] (-1,4) -- (1,4) node[right=5pt]{$S_1$};
        \draw[line width=1mm] (-2,2) -- (2,2) node[right=5pt]{$S_2$};
        \draw[line width=1mm] (-3,0) -- (3,0) node[right=5pt]{$S_3$};
        \draw (-0.5,1) node{$\mathcal{P}_{2,3}$}; \draw (-0.5,3)
        node{$\mathcal{P}_{2,2}$};
        \end{tikzpicture}
        \caption{The generations $0$ through $S_3$ of the BPRE $Z$ and the
        partition of the corresponding population into parts
        $\mathcal{P}_{i,j}$, $i,j=1,2,3$. The bold horizontal lines
        represent particles in the generations $S_1$, $S_2$ and $S_3$,
        that is, those comprising the groups $\mathcal{P}_{3,i}$,
        $i=1,2,3$. By definition, $\mathcal{P}_{2,1}=\oslash$.}
        \label{fig:population}
\end{figure}

Finally, we explain how we are going to treat the second summand
in \eqref{eq:dec}. We represent $T_n-T_{S_{\nu(n)-1}}$ as the sum
of two components: the times spent by
$(X_k)_{k=T_{S_{\nu(n)-1}+1},\ldots, T_n}$ in $(-\infty,
S_{\nu(n)-1})$ and $[S_{\nu(n)-1}, n]$, respectively. We shall
prove in Lemmas \ref{negligible(01)} and \ref{negligible1} below
that under the assumptions of our main theorems, the first
component is asymptotically negligible. Before presenting our
reasoning for the second component we find it convenient to recall
a few classical notions and formulate a technical lemma.

Let $D$ denotes the Skorokhod space of right-continuous functions
defined on $[0,\infty)$ with finite limits from the left at
positive points. The two commonly used topologies that the
Skorokhod space $D$ is equipped with are $J_1$- and
$M_1$-topologies. We refer to \cite{Billingsley:1968,
Jacod+Shiryaev:2003} and \cite{Whitt:2002} for comprehensive
accounts of the $J_1$- and the $M_1$-topologies, respectively. In
the sequel $\overset{{\rm J_1}}{\Rightarrow}$ and $\overset{{\rm
M_1}}{\Rightarrow}$ will mean weak convergence on $D$ when endowed
with the $J_1$-topology and the $M_1$-topology, respectively. Our
main results are one-dimensional distributional limit theorems.
However, we find it useful to appeal, at some intermediate steps,
to functional limit theorems on $D$. Working in this more general
setting simplifies considerably proofs of limit theorems involving
compositions. Theorem 13.2.2 in \cite{Whitt:2002} stated as Lemma
\ref{whi} below provides a necessary technical background.
\begin{lem}\label{whi}
Let $k\in\N$. The composition mapping $((x_1,\ldots, x_k),
\psi)\mapsto (x_1\circ \psi,\ldots, x_k\circ \psi)$ is $J_1$-
continuous at vectors $(x_1,\ldots, x_k)\in D^k$ with nonnegative
coordinates and nonnegative continuous and strictly increasing
$\psi$.
\end{lem}

Let $(X_k^\prime)_{k\in\N_0}$ be a starting at zero simple random
walk with reflection to the right at the origin, that is,
$X_0^\prime=0$, $\P\{X_{k+1}^\prime=i \pm 1| X_k^\prime = i\}
=1/2$ for $k,i\in\N$ and $\P\{X_{k+1}^\prime=1|X_k^\prime =0\}=1$
for $k\in\N_0$. We shall assume that $(X_k^\prime)_{k\in\N_0}$ is
independent of $(\xi_j,\rho_j)_{j\in\N}$ and $Z$. Set
\begin{equation}\label{eq:tprime}
T_n^\prime:=\inf\{k\in\N_0: X_k^\prime=n\}, \quad n\in\N_0.
\end{equation}
With this notation at hand we observe that
\begin{align}
&\text{given}~\{S_{\nu(n)-1}=j\}~\text{the time spent by
}~(X_k)_{k=T_{S_{\nu(n)-1}+1},\ldots, T_n}~\text{in}~
[S_{\nu(n)-1},n]\notag\\& \text{has the same distribution as}~
T_{n-j}^{\prime}.\label{eq:dec2}
\end{align}

It is well-known that
\begin{equation*}
n^{-1/2}X^\prime_{[n\cdot]}~\overset{{\rm
J_1}}{\Rightarrow}~ B(\cdot),\quad n\to\infty,
\end{equation*}
where $B:=(B(t))_{t\geq 0}$ is a reflected Brownian motion. By a
standard inversion argument, this yields
\begin{equation}\label{weak8}
n^{-2}T^\prime_{[n\cdot]}~\overset{{\rm
J_1}}{\Rightarrow}~ M(\cdot),\quad n\to\infty,
\end{equation}
where $M(t):=\inf\{s>0: B(s)=t\}$ for $t\geq 0$. By Proposition
3.7 on p.~71 in~\cite{Revuz+Yor:1999}
$$\E\exp(-sM(t))=\frac{1}{\cosh(t\sqrt{2s})},\quad s\geq 0.$$
Recalling \eqref{varthe} we conclude that $M(1)\od 2\vartheta$.
These facts explain the appearance of $\vartheta$
in Theorem \ref{thm:main11T}.

\section{The distribution tail behavior of $S_{\tau_1}$ and $\W_{\tau_1}$}\label{sec:tails}

To prove our main results we have to know the asymptotics of
$\P\{S_{\tau_1}>t\}$, $\P\{\W_{\tau_1}>t\}$ and $\P\big\{
S_{\tau_1} > g(t)x_1 ,\overline \W_{\tau_1} > f(t)x_2\big\}$ as
$t\to\infty$ for suitable functions $f$ and $g$.

\subsection{The marginal behavior}

\begin{lem}\label{future}
Assume that condition $(\xi)$ holds for $\beta\in(0,1)$ and that
$\E\log\rho\in [-\infty, 0)$. Then
\begin{equation}\label{eq:k34}
\P\big\{S_{\tau_1} > t\}~ \sim~ (\E
\tau_1)\P\{\xi>t\}~\sim~(\E\tau_1) t^{-\beta}\ell(t),\quad
t\to\infty
\end{equation}
and
\begin{equation}\label{eq:k35}
\P\big\{S_{\tau_1} > t\}~\sim~\P\{\max_{1\leq k\leq
\tau_1}\,\xi_k>t\},\quad t\to\infty.
\end{equation}
\end{lem}
\begin{proof}
The assumptions guarantee $\E\log\xi<\infty$ which in turn secures
$\E\tau_1<\infty$ by Lemma~\ref{lem:nu}.

The random variable $\tau_1$ {\it does not depend
on the future of the sequence} $(\xi_i)_{i\in\N}$, that is, for
each $n\in\N$, the collections of random variables
\begin{equation*}
(\xi_1,\ldots, \xi_n, {\bf 1}_{\{\tau_1 \le n\}})\quad  \mbox{and}
\quad (\xi_{n+1}, \xi_{n+2},\ldots)
\end{equation*}
are independent. With this at hand a specialization of Theorem 1
in \cite{Korshunov:2009} yields
$$\E (S_{\tau_1}\wedge t)~\sim~ (\E\tau_1)\E(\xi\wedge t),\quad
t\to\infty.$$ By Karamata's theorem (Proposition 1.5.8 in
\cite{bingham1989regular}) $$\E(\xi\wedge t)~\sim~
(1-\beta)^{-1}t^{1-\beta}\ell(t),\quad t\to\infty.$$ An
application of the monotone density theorem (Theorem 1.7.2 in
\cite{bingham1989regular}) completes the proof of
\eqref{eq:k34}.

Turning to the proof of \eqref{eq:k35}, write
\begin{eqnarray*}
\P\{\max_{1\leq k\leq \tau_1}\,\xi_k>t\}&=&\sum_{k\geq
1}\P\{\max_{1\leq i\leq k-1}\,\xi_i\leq t, \xi_k>t, \tau_1\geq
k\}\\&=& \P\{\xi>t\}\sum_{k\geq 1}\P\{\max_{1\leq i\leq
k-1}\,\xi_i\leq t, \tau_1\geq k\},
\end{eqnarray*}
where the last equality is a consequence of the fact that $\tau_1$
does not depend on the future of $(\xi_i)_{i\in\N}$. By the
dominated convergence theorem
$${\lim}_{t\to\infty}\frac{\P\{\max_{1\leq k\leq \tau_1}\,\xi_k>t\}}{\P\{\xi>t\}}= \E\tau_1.$$
\end{proof}

It is worth mentioning that Theorem 1 in \cite{Korshunov:2009}
cited in the previous proof treats standard random walks with
two-sided jumps of infinite mean stopped at an arbitrary random
variable of finite mean which does not depend on the future of the
sequence of jumps. In particular, the regular variation of the
distribution tail of a jump is not assumed.

Below we present a collection of auxiliary results borrowed from
Section 5 in \cite{BurDysIksMarRoi:2018+} which will be used in
the sequel.
\begin{lem}[Lemma 5.1 in \cite{BurDysIksMarRoi:2018+}] \label{lem:lu1}
Assume that \eqref{eq:reg_var_xi} holds with some $\beta>0$. Then
$$\P\{\W^0_1 >t\}~\sim~ (\E\vartheta^{\beta/2}) t^{-\beta/2}\ell(t^{1/2}),\quad t\to\infty,$$
where $\vartheta$ is a random variable with the Laplace transform
given in \eqref{varthe}.
\end{lem}
\begin{lem}[Lemma 5.2 in \cite{BurDysIksMarRoi:2018+}]\label{lem:zmom}
Assume that $\E\log\rho\in [-\infty, 0)$ and that, for some $s \le
2$, $\E (\rho \xi)^{s}$ and $\E \xi^s$ are finite. Then $\E
\Z_1^s<\infty$ and there exists a positive constant $C$ such that,
for all $n\in\N$,
\begin{equation*}
\E \Z_n^s \leq \left\{  \begin{array}{cc}
C& \mbox{ if } \gamma <1,\\
Cn& \mbox{ if } \gamma =1,\\
C\gamma^n& \mbox{ if } \gamma >1,
\end{array}\right.
\end{equation*}
where $\gamma=\E\rho^s$. If additionally $\E \xi^{2s}<\8$, then
\begin{equation*} \label{eq:lu1}
\E \W_1^s <\8.
\end{equation*}
\end{lem}
\begin{lem}[Lemma 5.4 in \cite{BurDysIksMarRoi:2018+}]\label{lem:Wnu}
Assume that, for some $s \le 2$, $\E \rho^s<1$, $\E (\rho \xi)^s$
and $\E \xi^s$ are finite. Then, for all $s_0 \in (0,s)$,
\begin{equation*}
\E  \left(\sum_{i=1}^{\tau_1}\Z_i  \right)^{s_0} < \infty.
\end{equation*}
If additionally $\E\xi^{3s/2}<\infty$, then
\begin{equation}\label{eq:5?}
\E\left(\sum_{i=1}^{\tau_1}\W_i^\downarrow\right)^{s_0}<\infty.
\end{equation}
\end{lem}
\begin{lem}[Lemma 5.5 in \cite{BurDysIksMarRoi:2018+}] \label{prop:main2}
Assume that $(\rho 1)$ holds for some $\alpha\in (0,2]$,
$\E\xi^{3\alpha/2}<\infty$ and $\E(\rho\xi)^\alpha<\infty$. Then
$$\P\bigg\{ \sum_{k=1}^{\tau_1} \big( \Z_k + \W_k^\downarrow  \big)>t\bigg\}~\sim~(\E\tau_1) \mathcal{C}_Z(\alpha) t^{-\a},\quad t\to\infty$$
for a positive constant $\mathcal{C}_Z(\alpha)$ which can be
represented as follows: $$\mathcal{C}_Z(\alpha) = \lim_{A \to
\infty} \E \ZZ^\alpha_{\sigma_A} {\bf 1}_{\{ \sigma_A < \tau_1 \}}
\cdot \lim_{x \to \infty}x^\alpha \P \bigg\{ \sum_{k\geq 0}\rho_1
\ldots \rho_k \xi_{k+1}>x \bigg\}.$$ Here, $\sigma_A = \inf \{ i
\in \NN \: : \: Z_j>A \mbox{ for some } j \leq S_i \}$. Both
limits exist and are finite.
\end{lem}
The assertion regarding the form of $\mathcal{C}_Z(\alpha)$ can be
derived from the proof of Lemma 5.5 in
\cite{BurDysIksMarRoi:2018+}. Note that an explicit expression for
$\mathcal{C}_Z(\alpha)$ is not known.

\begin{lem}[Proposition 5.7 in \cite{BurDysIksMarRoi:2018+}]\label{prop:tail_main}
The following asymptotic relations hold.
\begin{itemize}
\item[{\rm (C1)}] If $(\rho 1)$ holds for some $\alpha\in(0,2]$, either $\E\xi^{2\alpha}<\infty$ or~\eqref{eq:reg_var_xi}
holds with $\beta=2\alpha$, $\lim_{t\to\infty}\ell(t)=0$, and
$\E(\rho\xi)^\alpha <\infty$, then
$$\P\{\overline{\W}_{\tau_1}>t\}~\sim~ (\E\tau_1) \mathcal{C}_Z(\alpha)t^{-\alpha},\quad t\to\infty,$$
where $\mathcal{C}_Z(\alpha)$ is the same constant as in Lemma
\ref{prop:main2}.
\item[{\rm (C2)}] If $(\rho 1)$ holds for some $\alpha\in(0,2)$, \eqref{eq:reg_var_xi} holds with $\beta=2\alpha$
and $\lim_{t\to\infty}\ell(t)=\mathcal{C}_{\ell}\in(0,\infty)$,
$\E\rho^{\alpha+\varepsilon}<\infty$ and $\E\rho^\alpha
\xi^{\alpha+\varepsilon}<\infty$ for some $\varepsilon>0$, then
$$\P\{\overline{ \W}_{\tau_1}>t\}~\sim~ (\E \tau_1)\left((\E \vartheta^{\alpha})\mathcal{C}_{\ell}+\mathcal{C}_Z(\alpha)\right) t^{-\alpha},
\quad t\to\infty.$$
\item[{\rm (C3)}] If $(\rho 1)$ holds for some $\alpha\in(0,2]$, \eqref{eq:reg_var_xi} holds with $\beta=2\alpha$
and $\lim_{t\to\infty}\ell(t)=\infty$, and
$\E(\rho\xi)^\alpha<\infty$, then
$$\P\{\overline{\W}_{\tau_1}>t\}~\sim~ (\E\tau_1)(\E
\vartheta^{\alpha})t^{-\alpha}\ell(t^{1/2}),\quad t\to\infty.$$
\item[{\rm (C4)}] If $(\rho 2)$ holds, \eqref{eq:reg_var_xi} holds for some $\beta\in(0,4)$ such that $\beta/2\in\mathcal{I}$
and $\E (\rho\xi)^{\beta/2+\varepsilon}<\infty$ for some
$\varepsilon>0$, then $$\P\{\overline{\W}_{\tau_1}>t\}~\sim~
(\E\tau_1)(\E \vartheta^{\beta/2})t^{-\beta/2}\ell(t^{1/2}),\quad
t\to\infty.$$
\end{itemize}
\end{lem}

\subsection{The joint behavior}

The asymptotic behavior of $t\P\big\{ S_{\tau_1} >
g(t)x_1,\overline \W_{\tau_1} > f(t)x_2 \big\}$ as $t\to\infty$ is
determined by the mutual interplay of the distributions of $\xi$
and $\rho$. While Proposition \ref{prop:wp1} treats the situation
in which the distribution of $\xi$ dominates, Proposition
\ref{prop:w1} is concerned with the case in which the
contributions of the distributions of $\xi$ and $\rho$ are
comparable.
\begin{prop} \label{prop:wp1}
Assume that the assumptions of Theorem \ref{thm:main11T} are
satisfied for $\beta\in(0,1]$, with the exception that condition
$(\xi\rho 1)$ is not required. Then, for $x_1, x_2>0$,
$$ \P\big\{S_{\tau_1} > t x_1, \overline{ \W}_{\tau_1} > t^2 x_2\big\}~ \sim~(\E \tau_1) \E\big[ \min\big(x_1^{-\beta},
x_2^{-\beta/2}\vartheta^{\beta/2}\big)\big] \ell(t)
t^{-\beta},\quad t\to\infty,$$ where a random variable $\vartheta$
has the Laplace transform given by \eqref{varthe}.
\end{prop}
\begin{prop}\label{prop:w1}
Assume that the assumptions of Theorem \ref{thm:main2T} are
satisfied for $\beta\in(0,1]$, with the exception that condition
$(\xi\rho 2)$ is not required. Then
$$\lim_{t\to \infty} t \P\big\{S_{\tau_1} > a(t) , \overline{ \W}_{\tau_1} >t^{1/\alpha} \big\} =0.$$
\end{prop}

Our proofs of both propositions rely on decomposition
\eqref{eq:wp3} and `the principle of one big jump' which is
commonly used when analyzing random variables with regularly
varying distribution tails. In view of \eqref{eq:k35} the random
variable $S_{\tau_1}$ takes a large value if and only if at least
one of $\xi_1,\xi_2,\ldots,\xi_{\tau_1}$ is large. We shall choose
a stopping time $T=T(t)$ such that $\xi_T \approx \max _{1\leq k
\leq \tau_1} \xi _k \approx S_{\tau_1}$ on the event
$\{\max_{1\leq k \leq \tau_1} \xi _k>t\big\}$ and then show that
$\W^0_T$ dominates all the other terms in decomposition
\eqref{eq:wp3}. According to \eqref{distr_conv} the variable
$\W^0_T$  should be of magnitude $t^2$ on the event $\{ \max
_{1\leq k \leq \tau_1} \xi _k>t\big\}$ (see Lemma \ref{lem:k2} for
more details). Summarizing, it is plausible that
\begin{equation}\label{eq:intuition_joint_asymp}
\P\big\{S_{\tau_1} > t x_1, \overline{ \W}_{\tau_1} > t^2 x_2
\big\}~\approx~\P\big\{\xi_T > t x_1, \W^0_T > t^2 x_2
\big\},\quad t\to\infty.
\end{equation}

The rigorous proofs of Propositions \ref{prop:wp1} and
\ref{prop:w1} are similar to the proof of Proposition 6.1 in
\cite{BurDysIksMarRoi:2018+}. However, since we need a joint,
rather than marginal, asymptotic behavior, the details are more
involved. We start with a lemma that provides the asymptotic
behavior of the right-hand side in
\eqref{eq:intuition_joint_asymp}.
\begin{lem}\label{lem:wp2}
Let $\varsigma$ be an integer-valued random variable independent
of $(\Wc_n)_{n\in \N_0}$ and such that
$$\P\{\varsigma>t\}~\sim~ t^{-\beta}\ell(t),\quad t\to\infty$$ for some $\beta>0$ and some $\ell$ slowly varying at $\infty$. Then, for $x_1, x_2>0$,
$$\P\big\{\varsigma > tx_1, \Wc_{\varsigma}\ > t^2 x_2\big\}~ \sim~ \E\big[\min\big(x_1^{-\beta}, x_2^{-\beta/2}\vartheta^{\beta/2}\big)\big]
\ell(t)t^{-\beta}, \quad t\to \infty,$$ where $\vartheta$ is a
random variable with the Laplace transform given in
\eqref{varthe}.
\end{lem}
\begin{proof}
Put $$v_x = \inf\{k\in \N:\; \Wc_k > x\}, \quad x>1.$$ Since
$\Wc_n$ is monotone, it diverges to $+\infty$ a.s. This ensures
that $v_x$ is finite a.s. For fixed $x_1, x_2>0$ and sufficiently
large $t$, $$\P\big\{\varsigma > tx_1, \Wc_{\varsigma}\ > t^2
x_2\big\}=\P\big\{\varsigma >\max(tx_1, v_{t^2x_2})\big\}=\E R
\big(\max(tx_1, v_{t^2x_2})\big),$$ where $R(y) = \P\{\varsigma >
y\}$ for $y>0$. An application of a standard inversion technique
to \eqref{distr_conv} yields
$$t^{-1/2}v_t ~ \dod~ \vartheta^{-1/2},\quad t\to\infty.$$ Hence,
$$t^{-1}\max(tx_1, v_{t^2 x_2})~\dod~ \max(x_1, x_2^{1/2}
\vartheta^{- 1/2}), \quad t\to\infty$$ and subsequently
$$\frac {R\big(\max(tx_1, v_{t^2 x_2})\big)}{R(t)}~\dod~
\big[\max(x_1, x_2^{1/2} \vartheta^{- 1/2})\big]^{-\beta} =
\min\big(x_1^{-\beta}, x_2^{-\beta/2}\vartheta^{\beta/2}\big),
\quad t\to\infty$$ having utilized the regular variation of $R$.
Write $$\frac {R\big(\max(tx_1, v_{t^2 x_2})\big)}{R(t)}\leq \frac
{R(tx_1)}{R(t)}+\frac{R\big(v_{t^2x_2})\big)}{R\big(tx_2^{1/2}\big)}\frac
{R\big(tx_2^{1/2}\big)}{R(t)}.$$ It is shown in the proof of
Proposition 6.1 in \cite{BurDysIksMarRoi:2018+} that the family
$\Big(\frac{R(v_{t^2x_2})}{R(tx_2^{1/2})}\Big)_{t\geq t_0}$ is
uniformly integrable for large enough $t_0>0$. This in combination
with Potter's bound for regularly varying functions (Theorem
1.5.6(iii) in \cite{bingham1989regular}) enables us to conclude
that the family $\Big(\frac{R(\max(tx_1, v_{t^2
x_2}))}{R(t)}\Big)_{t\geq t_1}$ is uniformly integrable for large
enough $t_1>0$. Therefore,
$$\frac{\P\{\varsigma>tx_1, \Wc_\varsigma > t^2 x_2\}}{\P\{\varsigma>t\}} = \frac{\E R (\max(tx_1, v_{t^2
x_2}))}{R(t)}~\to~ \E\big[ \min\big(x_1^{-\beta},
x_2^{-\beta/2}\vartheta^{\beta/2}\big) \big],\quad t\to\infty$$
which completes the proof.
\end{proof}

Some parts of the proofs of Propositions
\ref{prop:wp1} and \ref{prop:w1} can be treated along similar
lines. As a preparation, we prove an auxiliary result.
\begin{lem}\label{lem:aux}
Assume that either the assumptions of Theorem \ref{thm:main11T},
with the case (A) and the condition ($\xi\rho 1$) being excluded,
or Theorem \ref{thm:main2T}, with the condition ($\xi\rho 2$)
being excluded, are satisfied for $\beta\in (0,1]$. Let $\delta\in
(0,\alpha)$ and $b_1$, $b_2$ and $b_3$ be positive functions
diverging to $+\infty$. Then, as $t\to\infty$, uniformly in
$k\in\N$,
$$\P\bigg\{\xi_k>b_1(t), k\leq \tau_1, \sum_{i=1}^{k-1}(\Z_i + \W^{\downarrow}_i) > b_2(t)\bigg\}=O\big(\P\{\xi>b_1(t)\}b_2(t)^{-\alpha}\big);$$
$$\P \big\{\xi_k> b_1(t), \W^{\downarrow}_k>b_2(t)\big\}=O\big(\E\big[ \xi^\delta
{\bf 1}_{\{\xi>b_1(t)\}}] b_2(t)^{-\delta}\big)$$ and, uniformly
in $k=1,2,\ldots, [b_3(t)]$,
$$\P\bigg\{\xi_k >b_1(t), \: \Z_k +\sum_{j=1}^{\Z_k}Y_{j,S_k}^\ast >b_2(t)\bigg\}=
O\big(\P\{\xi>b_1(t)\}^{\varepsilon/(\alpha+\varepsilon)}b_2(t)^{-\alpha}
b_3(t)\big)$$ with the same $\varepsilon$ as defined in Theorems
\ref{thm:main11T} and \ref{thm:main2T}. Here, for $j\in\N$, $j\leq
\Z_k$, $Y_{j,S_k}^\ast$ denotes the total progeny of the $j$th
particle in the generation $S_k$.
\end{lem}
\begin{proof}
The first relation is justified as follows:
uniformly in $k\in\N$,
\begin{align*}
& \P\bigg\{\xi_k>b_1(t), k\leq \tau_1,
\sum_{i=1}^{k-1}(\Z_i + \W^{\downarrow}_i) > b_2(t)\bigg\} \\
&  = \P\{\xi_k>b_1(t)\} \P \bigg\{k\leq \tau_1,  \:  \sum_{i=1}^{k-1}(\Z_i + \W^{\downarrow}_i) >b_2(t)\bigg\} \\
& \leq  \P\{\xi>b_1(t)\} \P\bigg\{ \sum_{i=1}^{\tau_1} \big( \Z_i
+ \W_i^\downarrow
\big)>b_2(t)\bigg\}=O\big(\P\{\xi>b_1(t)\}b_2(t)^{-\alpha}\big),\quad
t\to\infty,
\end{align*}
where the second inequality follows from the fact that $\tau_1$
does not depend on the future of $(\xi_i)_{i\in\N}$, and the last
equality is a consequence of Lemma \ref{prop:main2}.

While treating the second relation we use a
representation
$$\W^\downarrow_1=0,\quad \W^{\downarrow}_k = \sum_{i=1}^{\Z_{k-1}}D_i^{(k)},\quad k\geq 2,~\P-\text{a.s.},$$
where $D_i^{(k)}$ is the number of progeny in the generations
$S_{k-1}+1,\ldots, S_k-1$ of the $i$th particle in the generation
$S_{k-1}$. For fixed $k\geq 2$, under $\P_\omega$, $D_1^{(k)}$,
$D_2^{(k)},\ldots$ are iid and independent of $\Z_{k-1}$, and one
can check that $\E_\omega[D_i^{(k)}] = \xi_k-1$. With this at hand
we write, for any $\delta\in (0,\alpha)$ and any $k\in\N$,
\begin{align*}
&\P \big\{\xi_k> b_1(t),  \W^{\downarrow}_k>b_2(t)\big\} \leq b_2(t)^{-\delta} \E\bigg[ {\bf 1}_{\{\xi_k >b_1(t)\}} \bigg(\sum_{i=1}^{\Z_{k-1}} D_i^{(k)} \bigg)^\delta  \bigg]\\
&=b_2(t)^{-\delta}\E\bigg[ {\bf
1}_{\{\xi_k>b_1(t)\}}\E_\omega\bigg[ \bigg(\sum_{i=1}^{\Z_{k-1}}
D_i^{(k)} \bigg)^\delta \bigg| \Z_{k-1}\bigg] \bigg]\\&\le
b_2(t)^{-\delta}\E\bigg[{\bf 1}_{\{\xi_k>b_1(t)\}}\E_\omega\bigg[ \sum_{i=1}^{\Z_{k-1}} D_i^{(k)}  \bigg| \Z_{k-1}\bigg]^\delta\bigg]\\
&\leq b_2(t)^{-\delta} \E\big[ \xi^\delta {\bf
1}_{\{\xi>b_1(t)\}}]\E[\Z_{k-1}^\delta],
\end{align*}
where the first line is obtained with the help of Markov's
inequality, and the penultimate line follows by an application of
the conditional Jensen's inequality. Since $\E\rho^\delta<1$, an
appeal to Lemma \ref{lem:zmom} yields $\sup_{k\geq
1}\E[\Z_k^\delta]<\infty$.

Turning to the analysis of the third relation we
first note that, for fixed $k\in\N$, $Y_{1,S_k}^\ast$,
$Y_{2,S_k}^\ast,\ldots$ are $\P_\omega$-independent of copies of
$Y_1$ which are $\P$-independent of $\xi_k$. Therefore, according
to Lemma 7.2 in \cite{BurDysIksMarRoi:2018+}, there exists a
(nonrandom) constant $A>0$ such that, for $x>0$,
\begin{equation}\label{eq:aux1}
\P\bigg\{\sum_{j=1}^{\Z_k} Y_{j,S_k}^\ast >x \bigg| \Z_k,
\xi_k\bigg\}\leq A \Z_k^\alpha x^{-\alpha}\quad \P-\text{a.s.}
\end{equation}
Also, it can be checked (see the proof of Lemma 5.2 in
\cite{BurDysIksMarRoi:2018+} for details) that, for $k\in\N$,
\begin{equation}\label{eq:aux2}
\E_\omega \Z_k=\rho_k\E_\omega \Z_{k-1}+\rho_k\xi_k\leq
(1+\E_\omega \Z_{k-1})\rho_k\xi_k,\quad \P-\text{a.s.}
\end{equation}
Here, the last inequality follows from $\xi_k\geq 1$ $\P$-a.s.
Write, for $k\in\N$,
\begin{align*}
&\P\bigg\{\xi_k >b_1(t), \: \Z_k +\sum_{j=1}^{\Z_k}Y_{j,S_k}^\ast >b_2(t)\bigg\}\\
& \leq  \P\bigg\{\xi_k > b_1(t), \: \Z_k
>2^{-1}b_2(t)\bigg\}+ \P\bigg\{\xi_k >b_1(t), \: \sum_{j=1}^{\Z_k}Y_{j,S_k}^\ast >2^{-1}b_2(t)\bigg\}  \\
&\leq  \E{\bf 1}_{\{\xi_k>b_1(t)\}} \P_\omega\bigg\{\Z_k
>2^{-1}b_2(t)\bigg\} + \E{\bf 1}_{\{\xi_k>b_1(t)\}}
\P\bigg\{\sum_{j=1}^{\Z_k} Y_{j,S_k}^\ast >2^{-1}b_2(t)\bigg| \Z_k, \xi_k\bigg\} \\
& \leq 2^\alpha b_2(t)^{-\alpha} \E {\bf
1}_{\{\xi_k>b_1(t)\}}\E_\omega \Z_k^\alpha +2^\alpha
Ab_2(t)^{-\alpha}\E {\bf 1}_{\{\xi_k>b_1(t)\}} \Z_k^\alpha\\
& =2^\alpha (1+A)b_2(t)^{-\alpha} \E {\bf
1}_{\{\xi_k>b_1(t)\}}\E_\omega \Z_k^\alpha \leq 2^\alpha
(1+A)b_2(t)^{-\alpha} \E {\bf 1}_{\{\xi_k>b_1(t)\}}(\E_\omega
\Z_k)^\alpha
\end{align*}
having utilized Markov's inequality for the third line, inequality
\eqref{eq:aux1} for the fourth and the conditional Jensen's
inequality (observe that $\alpha\in (0, 1/2]$) for the fifth.
Further, for $k=1,2,\ldots, [b_3(t)]$,
\begin{align*}
&\E {\bf 1}_{\{\xi_k>b_1(t)\}}(\E_\omega \Z_k)^\alpha\\
&\leq \E (1+\E_\omega\Z_{k-1})^\alpha \E {\bf 1}_{\{\xi_k>b_1(t)\}} (\rho_k \xi_k)^\alpha \\
&\leq
\Big(1+\E(\rho\xi)^\alpha\Big(\sum_{j=0}^{k-1}\E\rho^\alpha\Big)\Big)
\E {\bf 1}_{\{\xi_k>b_1(t)\}}(\rho_k\xi_k)^\alpha
\leq (1+\E(\rho\xi)^\alpha) k \E {\bf 1}_{\{\xi_k>b_1(t)\}}(\rho_k \xi_k)^\alpha\\
& \leq (1+\E(\rho\xi)^\alpha) b_3(t) \E {\bf 1}_{\{\xi>b_1(t)\}} (\rho \xi)^\alpha \\
& \leq (1+\E(\rho\xi)^\alpha)\Big(\E (\rho\xi)^{\alpha +
\varepsilon}\Big)^{\alpha/(\alpha+\varepsilon)}b_3(t)\P \{ \xi
>b_1(t)\}^{\varepsilon/(\alpha+\varepsilon)},
\end{align*}
where the second line follows from \eqref{eq:aux2}, and the last
line is obtained with the help of H\"{o}lder's inequality.
Combining pieces together completes the proof of the third
relation.
\end{proof}

Fix $x_1>0$ and define the stopping time
$$T= T(t) = \inf\{i\in\N:\xi_i > (t-t^{3/4})x_1\},$$
where, as usual, $\inf \varnothing = \infty$. Put $$\W^0 =
\sum_{i=1}^{\tau_1} \W^0_i$$ and note that, for $i\in\N$, given
$\xi_i$,
\begin{equation}\label{eq:dist_equality_crit2}
\W_i^0~\od~ \Wc_{\xi_i-1},
\end{equation}
where $\Wc_0=0$, $\Wc_n$ for $n\in\N$ is the total progeny in the
first $n$ generations of $\Zc$ and $\xi_i$ is assumed independent
of $(\Wc_n)_{n\in\N_0}$. As a consequence, the random variables
$\W^0_1$, $\W^0_2,\ldots$ are identically distributed. Also, it is
clear that they are independent.
\begin{lem} \label{lem:k2}
Under the assumptions of Proposition \ref{prop:wp1} there exists a
constant $C$ such that, for $x_1, x_2>0$, as
$t\to\infty$,
\begin{multline*}
\P\big\{S_{\tau_1} > t x_1, \W^0 > t^2 x_2 \big\}=
\P\big\{S_{\tau_1} > t x_1, T\le \tau_1 < C \log t, \W^0_{T}
> (t - t^{3/4})^2 x_2 \\ \mbox{ and } \xi_i\leq t^{2/3}x_1,
\W^0_i\leq t^{5/3}x_2 \mbox{ for all } i\not= T  \big\} +
o(t^{-\beta}).
\end{multline*}
\end{lem}
\begin{proof}
The assumptions ensure that $\E\rho^\varepsilon<\infty$ and
$\E\xi^\varepsilon<\infty$ for some $\varepsilon>0$. Hence, $\E
e^{r\tau_1}<\infty$ for some $r>0$ by Lemma \ref{lem:nu}.

We start by proving a similar statement for the first coordinate
alone: as $t\to\infty$,
\begin{equation}\label{eq:w56}
\P\{S_{\tau_1} >tx_1\}= \P\big\{S_{\tau_1} >tx_1, T\le \tau_1 < C
\log t   \mbox{ and } \xi_i\leq t^{2/3}x_1 \mbox{ for all } i\not=
T \big\} + o(t^{-\beta})
\end{equation}
for any $C>2\beta/r$. By Markov's inequality
\begin{equation}\label{eq:w561}
\P\{ \tau_1 \ge C \log t\} \le t^{-Cr}\E e^{r\tau_1}=
o(t^{-2\beta}),\quad t\to\infty.
\end{equation}
Further, since
\begin{multline*}
\P\big\{ S_{\tau_1} > tx_1, \tau_1 < C \log t \ \mbox{ and } \
\xi_i > t^{2/3}x_1, \xi_j > t^{2/3}x_1 \ \mbox{ for some } i< j
\leq \tau_1 \big\} \\ \le 2^{-1}C^2 (\log t)^2 \P\big\{ \xi >
t^{2/3}x_1\big\}^2  =  o(t^{-\beta}),
\end{multline*}
we conclude that, as $t\to\infty$,
\begin{multline*}
\P\{ S_{\tau_1} >tx_1\}=\P\big\{ S_{\tau_1} > tx_1,   \tau_1 < C\log t,\xi_j > (t-t^{3/4})x_1 \text{ for some unique } j\le \tau_1\\
\text{ and }\xi_i\leq t^{2/3}x_1\text{ for all }i\leq \tau_1,i\neq
j\big\} + o(t^{-\beta})
\end{multline*}
because the sum $S_{\tau_1}$ must exceed $tx_1$. Therefore, on the
event $\{S_{\tau_1}
> tx_1\}$ with $t$ large enough we have $T= j\leq \tau_1$ and thereupon
$T\leq \tau_1$ which in turn yields \eqref{eq:w56}.

Analogously we can prove that, as $t\to\infty$,
\begin{multline*}
\P\big\{S_{\tau_1} > t x_1, \W^0 > t^2 x_2 \big\}
\\=\P\big\{S_{\tau_1} > t x_1, \W^0 > t^2 x_2, T \le \tau_1 < C \log t \mbox{ and } \xi_i\leq t^{2/3}x_1  \mbox{ for all } i\not= T  \big\}
+ o(t^{-\beta}).
\end{multline*}
Choosing $s>3\beta$ and appealing to~\eqref
{eq:dist_equality_crit2} and \eqref{eq:w57} yields
\begin{align*}
& \P\big\{\tau_1 < C \log t, \xi_i\leq t^{2/3}x_1, \W^0_i > t^{5/3}x_2 \mbox{ for some } i \le \tau_1 \big\} \\
& \le \sum_{i=1}^{[C \log t]} \P\big\{\xi_i\leq t^{2/3}x_1, \W^0_i > t^{5/3}x_2\big\} \le C \log t \P\{ \Wc_{[t^{2/3}x_1]} > t^{5/3}x_2\big\} \\
&\le C x_2^{-s} \log t \frac{\E(
\Wc_{[t^{2/3}x_1]})^s}{t^{4s/3}} \cdot
\frac{t^{4s/3}}{t^{5s/3}} = o(t^{-\beta}),\quad t\to\infty
\end{align*}
having utilized Markov's inequality for the last line. Since each
$\W^0_i$ for $i\not=T$ does not exceed $t^{5/3}x_2$, the variable
$\W^0_T=\sum_{i=1}^{T-1}\W_i^0+\W^0_T$ can only be larger than
$t^2x_2$ for large $t$ provided that the summand $\W^0_T$ is
larger than $(t-t^{3/4})^2 x_2$. This completes the proof of the
lemma.
\end{proof}
\begin{proof}[Proof of Proposition \ref{prop:wp1}]
Our proof consists of two steps. First we show that, for $x_1,
x_2>0$,
\begin{equation}\label{eq:59}
\P\{S_{\tau_1} > t x_1, \W^0 > t^2 x_2 \}~ \sim~ (\E \tau_1)
\E\big[ \min\big(x_1^{-\beta},
x_2^{-\beta/2}\vartheta^{\beta/2}\big)\big]
\ell(t)t^{-\beta},\quad t\to\infty
\end{equation}
and then that
\begin{equation}\label{eq:60}
\P\{S_{\tau_1} > t x_1, \overline{\W}_{\tau_1} - \W^0 > t^2 x_2
\}~ =~o\big( \ell(t)t^{-\beta}\big),\quad t\to\infty.
\end{equation}
{\sc Proof of \eqref{eq:59}.} Due to Lemma \ref{lem:k2} we are
left with investigating the asymptotic behavior of
\begin{multline*}
P(t) = \P\Big\{S_{\tau_1} > t x_1, T\le \tau_1 < C \log t,
\W^0_{T} > (t - t^{3/4})^2 x_2 \\  \mbox{ and } \xi_i\leq
t^{2/3}x_1, \W^0_i \leq t^{5/3}x_2 \mbox{ for all } i\not= T
\Big\},
\end{multline*}
where $C$ is a constant.

We need a more general version of the observation
made in the proof of Lemma \ref{future}: for each $n\in\N$, the
families $((\xi_k,\W^0_k)_{k\le n}, {\bf 1}_{\{\tau_1 \le n\}})$
and $(\xi_k,\W^0_k)_{k> n}$ are independent, that is, the random
variable $\tau_1$ does not depend on the future of the sequence
$(\xi_i,\W^0_i)_{i\in\N}$. This in combination with distributional
equality \eqref{eq:dist_equality_crit2} and Lemma \ref{lem:wp2}
yields
\begin{align*}
P(t) & \le \sum_{j\ge 1} \P\left\{\xi_j > (t - t^{3/4})x_1,
\W^0_{j} > (t - t^{3/4})^2 x_2,
\tau_1 \ge j \right\}\\
& = \sum_{j\ge 1} \P\left\{\xi_j > (t - t^{3/4})x_1,   \W^0_{j} > (t - t^{3/4})^2 x_2 \right\} \P \left\{ \tau_1 \ge j \right\}\\
& =  \P\left\{\xi_1 > (t - t^{3/4})x_1,   \Wc_{\xi_1-1} > (t - t^{3/4})^2 x_2 \right\}  \sum_{j\ge 1}  \P \left\{ \tau_1 \ge j \right\}\\
&\sim (\E \tau_1) \E\big[
\min\big(x_1^{-\beta}, x_2^{-\beta/2}\vartheta^{\beta/2}\big)\big]
t^{-\beta}\ell(t),\quad t\to\infty.
\end{align*}
To obtain a lower bound for $P(t)$ we shall use the following
inequality
\begin{equation}\label{eq:aux}
\P\{\xi_1\leq t^{2/3}x_1, \W^0_1 \leq t^{5/3}x_2\} \ge 1 -
\P\{\xi_1> t^{2/3}x_1\}-\P\{\W^0_1>t^{5/3}x_2\}\geq 1-t^{-\beta/3}
\end{equation}
which holds by virtue of \eqref{eq:reg_var_xi} and Lemma
\ref{lem:lu1} for $t$ large enough. Recalling
\eqref{eq:dist_equality_crit2} and appealing again to the fact
that $\tau_1$ does not depend on the future of the sequence
$\{(\xi_k,\W^0_k)\}_{k \in \NN}$ we obtain
\begin{align*}
P(t) &\ge \sum_{j\geq 1} \P\big\{ j \le \tau_1 < C \log t,  \xi_j > t x_1,   \W^0_{j} > t^2 x_2\\
&\hspace{4cm} \mbox{ and } \xi_i \leq t^{2/3}x_1, \W^0_i\leq t^{5/3}x_2  \mbox{ for all } i\not= j  \big\}\\
&\ge \sum_{j\geq 1} \P\big\{j\leq \tau_1<C\log t, \xi_i\leq t^{2/3}x_1, \W^0_i\leq t^{5/3}x_2\mbox{ for all } i < j \big\}\\
& \qquad \qquad \P\big\{\xi_j > t x_1, \W^0_{j} > t^2 x_2 \big\} \P\big\{\xi_i\leq t^{2/3}x_1, \W^0_i\leq t^{5/3}x_2
\mbox{ for all } j < i < C\log t \big\}\\
&\ge \P\big\{ \xi{_1} > t x_1,   \Wc_{\xi{_1}{-1}} > t^2 x_2 \big\} (1 -  t^{-\beta/3})^{C \log t}\\
&\qquad\cdot \sum_{j\geq 1} \P\big\{j\leq \tau_1<C\log t,\xi_i\leq
t^{2/3}x_1, \W^0_i\leq t^{5/3}x_2  \mbox{ for all } i < j \big\}.
\end{align*}
In view of Lemma \ref{lem:wp2} it remains to note that
$${\lim\inf}_{t\to\infty} \sum_{j\geq 1} \P\big\{j\leq \tau_1<C\log t,\xi_i\leq t^{2/3}x_1, \W^0_i\leq
t^{5/3}x_2  \mbox{ for all } i < j \big\}\geq \E \tau_1$$ by
Fatou's lemma.

\medskip

\noindent {\sc Proof of \eqref{eq:60}.} In the case (A) of Theorem
\ref{thm:main11T} relation \eqref{eq:60} is just a consequence of
Lemma \ref{lem:Wnu} and Markov's inequality:
\begin{align*}
&\P\big\{ S_{\tau_1} > t x_1, \overline{\W}_{\tau_1} - \W^0 > t^2
x_2 \big\}\\ & \le \P\bigg\{\sum_{i=1}^{\tau_1}(\Z_i +
\W^{\downarrow}_i) > t^2 x_2 \bigg\} \\
& \le x_2^{-(\beta+\gamma)/2} \E \bigg(\sum_{i=1}^{\tau_1}(\Z_i +
\W^{\downarrow}_i) \bigg)^{(\beta + \gamma)/2} \cdot
t^{-(\beta+\gamma)}\\&\leq
x_2^{-(\beta+\gamma)/2}\Big(\E
\bigg(\sum_{i=1}^{\tau_1}\Z_i\bigg)^{(\beta + \gamma)/2}+ \E
\bigg(\sum_{i=1}^{\tau_1}\W^{\downarrow}_i\bigg)^{(\beta+\gamma)/2}\Big)
\cdot t^{-(\beta+\gamma)},
\end{align*}
where $\gamma>0$ is small enough (in particular,
$\beta+\gamma<2$), and the last inequality is justified by
subadditivity of $s\mapsto s^{(\beta+\gamma)/2}$ for $s\geq 0$.

Assume from now on that either the case (B1) or (B2) of Theorem
\ref{thm:main11T} prevails. Arguing as in the proof of Lemma
\ref{lem:k2} one can check that it is sufficient to show that, for
any $x_1, x_2>0$ and a constant $C>0$,
\begin{equation}\label{eq:70}
\P\big\{ \xi_T > t x_1, T \leq \tau_1 \leq C\log t, \:
\overline{\W}_{\tau_1} - \W^0 > t^2 x_2 \big\}~ =~o\big( \ell(t)
t^{-\beta}\big),\quad t\to\infty.
\end{equation}
Observe that decomposition \eqref{eq:wp3} implies that on the
event $ \{ T \leq \tau_1 \}$
\begin{equation*}
\overline{\W}_{\tau_1} - \W^0 = \sum_{i=1}^{T-1}(\Z_i +
\W^{\downarrow}_i) + \W^{\downarrow}_T+
\Z_T+\sum_{j=1}^{\Z_T}
Y_{j,S_T}^\ast+ \sum_{k=S_T+1}^{S_{\tau_1}}Y_k,
\end{equation*}
where, for $j\in\N$, $j\leq \Z_T$, $Y_{j,S_T}^\ast$ denotes the
total progeny of the $j$th particle in the generation $S_T$. Thus,
to ensure \eqref{eq:70} it is sufficient to check that, as
$t\to\infty$,
\begin{align*}
I_1(t) & =  \P\bigg\{\xi_T > t x_1, T \leq \tau_1<C\log t , \: \sum_{i=1}^{T-1}(\Z_i + \W^{\downarrow}_i) > t^2 x_2 \bigg\}= o\big( \ell(t) t^{-\beta}\big),\\
I_2(t) & =  \P\big\{\xi_T > t x_1, T \leq \tau_1<C\log t, \: \W^{\downarrow}_T > t^2 x_2 \big\} = o\big( \ell(t) t^{-\beta}\big),\\
I_3(t) & =  \P\bigg\{\xi_T > t x_1, T \leq \tau_1<C\log t, \: \Z_T +\sum_{j=1}^{\Z_T}Y_{j,S_T}^\ast > t^2 x_2 \bigg\}=o\big( \ell(t) t^{-\beta}\big),\\
I_4(t) & =  \P\bigg\{\xi_T > t x_1, T \leq \tau_1<C\log t, \:
\sum_{k=S_T+1}^{S_{\tau_1}}Y_k > t^2 x_2 \bigg\}=o\big( \ell(t)
t^{-\beta}\big).
\end{align*}
To treat $I_1(t)$, $I_2(t)$ and $I_3(t)$ we use
Lemma \ref{lem:aux} with $b_1(t)=tx_1$, $b_2(t)=t^2x_2$ and
$b_3(t)=C\log t$. Recalling that $\alpha=\beta/2$ we obtain
\begin{align*}
I_1(t) & \leq \sum_{k=1}^{[C\log t]} \P\bigg\{\xi_k>tx_1, \:
\max_{1\leq i \leq k-1}\,\xi_i \leq (t-t^{3/4})x_1, k\leq \tau_1,
\sum_{i=1}^{k-1}(\Z_i + \W^{\downarrow}_i) > t^2 x_2\bigg\}\\
&=O\big(\P\{\xi>tx_1\}t^{-\beta}\log t\big)=o\big( \ell(t)
t^{-\beta}\big),\quad t\to\infty.
\end{align*}
Further, for any $\delta\in (0,\beta/2)$,
\begin{align*}
I_2(t)\leq \sum_{k=2}^{[C\log t]}  \P \big\{ \xi_k >tx_1,
\W^{\downarrow}_k  >t^2x_2 \big\}= O\big(\E\big[ \xi^\delta {\bf
1}_{\{\xi>tx_1\}}] t^{-2\delta}\log t\big),\quad t\to\infty.
\end{align*}
According to Karamata's theorem (Theorem 1.6.5 in
\cite{bingham1989regular}) the function $t\mapsto
\E\big[\xi^\delta {\bf 1}_{\{\xi
>tx_1\}}]\times t^{-2\delta}\log t$ is regularly varying at $\infty$ of index
$-\beta-\delta$, whence $I_2(t)=o(t^{-\beta}\ell(t))$ as
$t\to\infty$. Passing to $I_3(t)$ we infer
\begin{align*}
I_3(t) & \leq \sum_{k=1}^{[C\log t]} \P\bigg\{\xi_k > t x_1, \:
\Z_k +\sum_{j=1}^{\Z_k}Y_{j,S_k}^\ast > t^2 x_2
\bigg\}=O\big(\P\{\xi>tx_1\}^{\varepsilon/(\alpha+\varepsilon)}
t^{-\beta}(\log t)^2\big)\\&=o\big(t^{-\beta}\ell(t)\big),\quad
t\to\infty.
\end{align*}
Finally, the relation for $I_4(t)$ holds true just because $\xi_T$
and $\sum_{k=S_T+1}^{S_{\tau_1}}Y_k$ are independent.
\end{proof}
\begin{proof}[Proof of Proposition \ref{prop:w1}]
Recalling the asymptotic relation obtained in Lemma \ref{lem:lu1}
and arguing as in the proof of Lemma~\ref{future} we conclude that
\begin{equation*}
\P \Bigg\{\W^0=\sum_{k=1}^{\tau_1} \W_k^0 > t
\Bigg\}~ \sim~ (\E \tau_1) \P \{\W_1^0>t\},\quad t\to\infty
\end{equation*}
which implies
\begin{equation*}
\P \left\{\W^0>t^{1/\alpha} \right\}~ \sim~ (\E
\tau_1)(\E\vartheta^{\beta/2})\ell \big(t^{1/(2\alpha)}\big)
t^{-\beta/2\alpha} = o(t^{-1}),\quad t\to\infty.
\end{equation*}
Thus, it is sufficient to prove
\begin{equation*}
\P\big \{S_{\tau_1} > a(t), \overline{\W}_{\tau_1} - \W^0 >
t^{1/\alpha} \big\} = o(t^{-1}),\quad t\to\infty.
\end{equation*}

In view of $\PP \big\{ \max _{1\leq k \leq
\tau_1}\,\xi _k >t\big\} \sim \PP \{ S_{\tau_1} >t\}$ as
$t\to\infty$ (see \eqref{eq:k35}),
$$\P\{S_{\tau_1}>t, \max_{1\leq k\leq \tau_1}\,\xi_k\leq t\}=\P\{S_{\tau_1}>t\}-\P\{\max_{1\leq k\leq \tau_1}\,\xi_k>t\}=o(\P\{S_{\tau_1}>t\}),$$
whence
\begin{equation*}
\P\big \{ S_{\tau_1} > a(t), \; \max _{1\leq k \leq \tau_1}\, \xi
_k \leq a(t), \overline{\W}_{\tau_1} - \W^0 > t^{1/\alpha} \big\}
= o(t^{-1}),\quad t\to\infty.
\end{equation*}
As a consequence, we are left with showing that
\begin{equation*}
\P\big\{\max _{1\leq k \leq \tau_1}\, \xi _k > a(t),
\overline{\W}_{\tau_1} - \W^0 > t^{1/\alpha} \big\} =
o(t^{-1}),\quad t\to\infty.
\end{equation*}
By Lemma~\ref{lem:nu}, $\E\exp(r\tau_1)<\infty$
for some $r>0$. This implies that there exists $C>0$ such that
$\sum_{k > [C\log t]}\P\{\tau_1\geq k\}=o(t^{-1})$ as $t\to\infty$
and thereupon
\begin{multline*}
\P\big\{\max _{1\leq k \leq \tau_1} \xi _k > a(t),
\overline{\W}_{\tau_1} - \W^0 > t^{1/\alpha} \big\}\\
\leq \sum_{k=1}^{[C\log t]}\P\big\{ \xi _k > a(t),\: k\leq \tau_1,
\overline{\W}_{\tau_1} - \W^0 > t^{1/\alpha}\big\}+o(t^{-1}),\quad
t\to\infty.
\end{multline*}
To ensure that the first summand on the right-hand side is
$o(t^{-1})$ it is more than sufficient if we can check that, for
some $\gamma>0$,
\begin{equation*}
\P \{ \xi_k>a(t),\: k\leq \tau_1, \: \ \overline{\W}_{\tau_1} -
\W^0> t^{1/\alpha}\} = o(t^{-1-\gamma}),\quad t\to\infty
\end{equation*}
uniformly in $k=1,\ldots, [C\log t]$. The latter is accomplished
by making use of a decomposition similar to the one used in the
proof of Proposition~\ref{prop:wp1}, namely: on the event $\{k\leq
\tau_1\}$,
\begin{equation*}
\overline{\W}_{\tau_1} - \W^0 = \sum_{i=1}^{k-1}(\Z_i +
\W^{\downarrow}_i) + \W^{\downarrow}_k+
\Z_k+\sum_{j=1}^{\Z_k}Y_{j,S_k}^\ast +
\sum_{j=S_k+1}^{S_{\tau_1}}Y_j.
\end{equation*}
Summarizing, our task boils down to proving that, uniformly in
$k=1,\ldots, [C\log t]$, as $t\to\infty$,
\begin{align*}
J_1(k, t) & =  \P\bigg\{\xi_k > a(t), k \leq \tau_1, \: \sum_{i=1}^{k-1}(\Z_i + \W^{\downarrow}_i) > t^{1/\alpha} \bigg\}= o\big( t^{-1-\gamma}\big),\\
J_2(k, t) & =  \P\big\{ \xi_k > a(t), k \leq \tau_1, \: \W^{\downarrow}_k > t^{1/\alpha} \big\} = o\big( t^{-1-\gamma}\big),\\
J_3(k, t) & =  \P\bigg\{\xi_k > a(t), k \leq \tau_1, \: \Z_k +\sum_{{j=1}}^{\Z_k}Y_{j,S_k}^\ast > t^{1/\alpha} \bigg\}=o\big( t^{-1-\gamma}\big),\\
J_4(k, t) & =  \P\bigg\{\xi_k > a(t), k \leq \tau_1, \:
\sum_{j=S_k+1}^{S_{\tau_1}}{Y_j} > t^{1/\alpha}
\bigg\}=o\big(t^{-1-\gamma}\big).
\end{align*}
To prove the limit relations for $J_1(k, t)$, $J_2(k, t)$ and
$J_3(k, t)$ we apply Lemma \ref{lem:aux} with $b_1(t)=a(t)$,
$b_2(t)=t^{1/\alpha}$ and $b_3(t)=C\log t$. This enables us to
conclude that $J_1(k,t)=O\big(\P\{\xi>a(t)\}t^{-1}\big)=O(t^{-2})$
uniformly in $k\in\N$ as $t\to\infty$. Also, for any $\delta\in
(0,\alpha)$,
\begin{align*}
J_2(k, t)   \leq \P \big\{\xi_k > a(t),  \W^{\downarrow}_k  >
t^{1/\alpha} \big\}= O\big(t^{-\delta/\alpha } \E\big[ \xi^\delta
{\bf 1}_{\{\xi > a(t) \}}]\big),\quad t\to\infty
\end{align*}
uniformly in $k\in\N$. Invoking Karamata's theorem (Theorem 1.6.5
in \cite{bingham1989regular}) we infer that the function $t\mapsto
t^{-\delta/\alpha } \E\big[ \xi^\delta {\bf 1}_{\{\xi > a(t) \}}]$
is regularly varying at $\infty$ of index
$-1-\delta(\alpha^{-1}-\beta^{-1})<-1$, hence
$J_2(k,t)=o(t^{-1-\gamma})$ uniformly in $k\in\N$ as $t\to\infty$.
Further, as $t\to\infty$,
\begin{align*}
J_3(k,t)\leq \P\bigg\{\xi_k > a(t), \: \Z_k
+\sum_{j=1}^{\Z_k}Y_{j,S_k}^\ast > t^{1/\alpha}
\bigg\}=O(t^{-1-\varepsilon/(\alpha+\varepsilon)}\log t)=
o\big(t^{-1-\gamma}\big)
\end{align*}
uniformly in $k=1,2,\ldots, [C\log t]$. The asymptotic estimate
for $J_4(k, t)$ is justified by the independence of $\xi_k$ and
$\sum_{j=S_k+1}^{S_{\tau_1}}Y_j$.
\end{proof}

\section{The proofs}\label{sec:proof}

Recall the notation
\begin{equation*}
\nu(t) = \inf\{ n \in \N \: : \: S_n >t \},\quad t\geq 0.
\end{equation*}
Put
\begin{equation}\label{eq:RenFunDf}
U(t) = \E \nu(t) = \sum_{k\geq 0}\P\{ S_k \leq t \},\quad t\geq 0,
\end{equation}
so that $U$ is the renewal function. It is well-known (see, for
instance, formula (2.1) in \cite{Erickson:1970}) that
\eqref{eq:reg_var_xi} with $\beta\in (0,1]$ entails
\begin{equation}\label{ren1}
\lim_{t\to\infty}\frac{m(t)}{t}U(t)=(\Gamma(2-\beta)\Gamma(1+\beta))^{-1},
\end{equation}
where $\Gamma$ is the Euler gamma function and the function $m$ is
defined in \eqref{eq: mt}.

We start with several technical results. The relevance of the
random variables $Y_n$ defined in Lemma \ref{lem:quenchedY} is
justified by formula \eqref{decomp}.
\begin{lem}\label{lem:quenchedY}
Put
\begin{equation*}
Y_n =
\sum_{j=0}^{S_{\nu(n)-1}}\Big(U_j^{(n)}-U_j^{(S_{\nu(n)-1})}\Big),\quad
n\in\N.
\end{equation*}
Then
\begin{align*}
\E_\omega Y_n  =&  (n-S_{\nu(n)-1})\rho_1\rho_2\cdot\ldots\cdot
\rho_{\nu(n)}\\&+
(n-S_{\nu(n)-1})\rho_{\nu(n)}(\xi_1\rho_2\cdot\ldots\cdot\rho_{\nu(n)-1}+\ldots+\xi_{\nu(n)-1}).
\end{align*}
\end{lem}
\begin{proof}
Recurrence relation \eqref{eq: gr1} entails, for $k\in\N$ and
$j=0,1,\ldots, k-1$,
$$\E_\omega U^{(k)}_j=\sum_{i=j}^{k-1}\prod_{r=j}^i q_r\quad\text{a.s.},$$
where $q_r=(1-\omega_r)/\omega_r$ for $r\in\N_0$ and thereupon
\begin{eqnarray*}
\E_\omega
\sum_{j=0}^{S_{\nu(n)-1}}\Big(U_j^{(n)}-U_j^{(S_{\nu(n)-1})}\Big)&=&\sum_{j=0}^{S_{\nu(n)-1}}\bigg(\sum_{i=j}^{n-1}\prod_{r=j}^i
q_r- \sum_{i=j}^{S_{\nu(n)-1}-1}\prod_{r=j}^i q_r\bigg)\\
&=&\sum_{j=0}^{S_{\nu(n)-1}}\sum_{i=S_{\nu(n)-1}}^{n-1}\prod_{r=j}^i
q_r\\&=&\sum_{i=S_{\nu(n)-1}}^{n-1}\prod_{r=0}^i
q_r+\sum_{m=1}^{\nu(n)-1}
\sum_{j=S_{m-1}+1}^{S_m}\sum_{i=S_{\nu(n)-1}}^{n-1}\prod_{r=j}^i
q_r\\&=&(n-S_{\nu(n)-1})\rho_1\rho_2\cdot\ldots\cdot
\rho_{\nu(n)}\\&+&
(n-S_{\nu(n)-1})\rho_{\nu(n)}(\xi_1\rho_2\cdot\ldots\cdot\rho_{\nu(n)-1}+\ldots+\xi_{\nu(n)-1}).
\end{eqnarray*}
\end{proof}

In Lemma \ref{lem:rhonu} we show that the contribution of
$\rho_{\nu(n)}$ is negligible in an appropriate sense. Recall
that, for $i=1,2$, the functions $c_i$ were defined in the
conditions ($\xi\rho\, {\rm i}$) in Section \ref{beta<1} when
$\beta\in (0,1)$ and in \eqref{eq: cc} when $\beta=1$.
\begin{lem}\label{lem:rhonu}
\begin{itemize}
\item[(i)]  Assume that ($\xi$) and ($\xi\rho1$) hold {for
$\beta\in (0,1]$} and that $\E\rho^\gamma<\infty$ for some
$\gamma>\beta/2$. Then $\rho_{\nu(n)}/c_1(n)\topr 0$ as
$n\to\infty$.
\item[(ii)] Assume that ($\xi$) and ($\xi\rho2$) hold {for
$\beta\in (0,1]$ and} some $\alpha\leq \beta/2$ and
$\E\rho^\gamma<\infty$ for some $\gamma>\alpha$. Then
$\rho_{\nu(n)}/c_2(n)\topr 0$ as $n\to\infty$.
\end{itemize}
\end{lem}
\begin{proof}
We first check that ($\xi\rho1$) in combination with ($\xi$)
entails
\begin{equation}\label{epsil}
\lim_{t\to\infty}\frac{\P\{\xi>t^{1/2}, \rho>\varepsilon
c_1(t)\}}{\P\{\xi>t\}}=0
\end{equation}
for all $\varepsilon>0$. It suffices to prove this for
$\varepsilon\in (0,2)$. Fix any such an $\varepsilon$ and pick
$t_0$ so large to ensure that $c_1(\varepsilon t/2)/c_1(t)\leq
\varepsilon$ for $t\geq t_0$. This is possible because $c_1$ is
regularly varying at $\infty$ of index $1$. Then, for $t\geq t_0$,
\begin{align*}
\frac{\P\{\xi>t^{1/2},\,\rho>\varepsilon c_1(t)\}}{\P\{\xi>t\}} & \leq \frac{\P\{\xi>t^{1/2},\,\rho>c_1(\varepsilon t/2)\}}{\P\{\xi>t\}}\\
& \leq \frac{\P\{\xi>(\varepsilon
t/2)^{1/2},\,\rho>c_1(\varepsilon t/2)\}}{\P\{\xi>\varepsilon
t/2\}}\frac{\P\{\xi>\varepsilon t/2\}}{\P\{\xi>t\}}.
\end{align*}
The right-hand side tends to zero as $t\to\infty$ in view of
($\xi\rho1$) and the regular variation of $s\mapsto \P\{\xi>s\}$.
This completes the proof of \eqref{epsil}.

As a consequence of~\eqref{ren1} we have
\begin{equation}\label{ren}
\P\{\xi>t\}U(t)=O(1),\quad t\to\infty.
\end{equation}
For any $\delta>0$,
\begin{eqnarray*}
\P\{\rho_{\nu(n)}>\delta
c_1(n)\}&=&\int_{[0,\,n-n^{1/2}]}\P\{\xi>n-y,\,\rho>\delta
c_1(n)\}{\rm d}U(y)\\&+&\int_{(n-n^{1/2},\,
n]}\P\{\xi>n-y,\,\rho>\delta c_1(n)\}{\rm d}U(y)=:I_1(n)+I_2(n).
\end{eqnarray*}
Further, $$I_1(n)\leq \frac{\P\{\xi>n^{1/2},\,\rho>\delta
c_1(n)\}}{\P\{\xi>n\}}\big(\P\{\xi>n\}U(n)\big)~\to~0,\quad
n\to\infty$$ by \eqref{epsil} and \eqref{ren}. As for $I_2(n)$ we
have
\begin{align*}
I_2(n)& \leq \P\{\rho>\delta c_1(n)\}(U(n)-U(n-n^{1/2}))\leq
\P\{\rho>\delta c_1(n)\}U(n^{1/2})\\& \leq (\E \rho^\gamma)(\delta
c_1(n))^{-\gamma}U(n^{1/2})
\end{align*}
by subadditivity of $U$ and Markov's inequality. According to
\eqref{ren1} $U$ is regularly varying at $\infty$ of index
$\beta$. Also, $c_1$ is regularly varying of index $1$. Therefore,
the right-hand side of the last centered formula converges to zero
as $n\to\infty$. This completes the proof of part (i).
\medskip
The proof of part (ii) is analogous. Therefore, we only discuss
principal steps. Similarly to \eqref{epsil} we have
\begin{equation}\label{epsil1}
\lim_{t\to\infty}\frac{\P\{\xi>t^{\alpha/\beta}, \rho>\varepsilon
c_2(t)\}}{\P\{\xi>t\}}=0
\end{equation}
for all $\varepsilon>0$. Let $\delta>0$. Using a decomposition
\begin{eqnarray*}
\P\{\rho_{\nu(n)}>\delta
c_2(n)\}&=&\int_{[0,\,n-n^{\alpha/\beta}]}\P\{\xi>n-y,\,\rho>\delta
c_2(n)\}{\rm d}U(y)\\& & +\int_{(n-n^{\alpha/\beta},\,
n]}\P\{\xi>n-y,\,\rho>\delta c_2(n)\}{\rm d}U(y)=:J_1(n)+J_2(n)
\end{eqnarray*}
we further obtain
$$J_1(n)\leq \frac{\P\{\xi>n^{\alpha/\beta},\,\rho>\delta c_2(n)\}}{\P\{\xi>n\}}\big(\P\{\xi>n\}U(n)\big)~\to~0,\quad n\to\infty$$ by \eqref{epsil1} and
\eqref{ren} and
\begin{align*}
J_2(n) & \leq \P\{\rho>\delta
c_2(n)\}(U(n)-U(n-n^{\alpha/\beta}))\leq \P\{\rho>\delta
c_2(n)\}U(n^{\alpha/\beta})\\& \leq (\E \rho^{\gamma})(\delta
c_2(n))^{-\gamma} U(n^{\alpha/\beta}).
\end{align*}
The right-hand side of the last centered formula converges to zero
as $n\to\infty$ because $U(t^{\alpha/\beta})$ and $c_2(t)^\gamma$
are regularly varying at $\infty$ of indices $\alpha$ and
$(\beta-\alpha)\alpha^{-1}\gamma$, respectively, and
$\alpha<(\beta-\alpha)\alpha^{-1}\gamma$. The latter is secured by
$\gamma>\alpha$ and $\beta\geq 2\alpha$.
\end{proof}

\subsection{The case $\beta \in (0,1)$}
\begin{lem}\label{negligible(01)}
Assume that $\E\log\rho\in [-\infty, 0)$ and that ($\xi$) holds
{for $\beta\in (0,1)$}.
\begin{itemize}
\item[(i)]  If ($\xi\rho1$) holds, and $\E\rho^\gamma<\infty$ for some $\gamma>\beta/2$, then
\begin{equation}\label{inter}
n^{-2}Y_n ~\topr~0,\quad n\to\infty.
\end{equation}
\item[(ii)] If ($\xi\rho2$) holds for some $\alpha\leq \beta/2$, and $\E\rho^\gamma<\infty$ for some $\gamma>\alpha$, then
\begin{equation}\label{inter000}
\P\{\xi>n\}^{1/\alpha}Y_n~\topr~0,\quad n\to\infty.
\end{equation}
\end{itemize}
\end{lem}
\begin{proof}
For part (i) it suffices to prove that
\begin{equation}\label{purp}
n^{-2}\E_\omega Y_n\topr 0,\quad n\to\infty
\end{equation}
because {while} $\P_\omega\{Y_n>\varepsilon
n^2\}\topr 0$ for all $\varepsilon>0$ as $n\to\infty$ then follows
by Markov's inequality, $\P\{Y_n>\varepsilon n^2\}=\E
\P_\omega\{Y_n>\varepsilon n^2\}\to 0$ for all $\varepsilon>0$ as
$n\to\infty$ is justified by the Lebesgue dominated convergence
theorem.

According to Lemma \ref{lem:quenchedY},
$$\E_\omega Y_n  = (n-S_{\nu(n)-1})\rho_1\cdot\ldots\cdot
\rho_{\nu(n)}+
(n-S_{\nu(n)-1})\rho_{\nu(n)}(\xi_1\rho_2\cdot\ldots\cdot\rho_{\nu(n)-1}+\ldots+\xi_{\nu(n)-1}).$$
In view of $\E\log\rho\in [-\infty, 0)$ we have $\lim_{n\to\infty}
\rho_1\cdot\ldots\cdot \rho_n=0$ a.s., whence $\lim_{n\to\infty}
\rho_1\cdot\ldots\cdot
\rho_{\nu(n)-1}=\lim_{n\to\infty}\rho_1\cdot\ldots\cdot
\rho_{\nu(n)}=0$ a.s.\ because $\lim_{n\to\infty}\nu(n)=\infty$
a.s. This in combination with $n-S_{\nu(n)-1}\leq n$ a.s.\ proves
that
$$(n-S_{\nu(n)-1})\rho_1\rho_2\cdot\ldots\cdot
\rho_{\nu(n)}=o(n),\quad n\to\infty~\text{a.s.}$$ By Lemma
\ref{lem:rhonu}, $n^{-1}\rho_{\nu(n)}\topr 0$ as $n\to\infty$.
Therefore, \eqref{purp} follows if we can show that
$\xi_1\rho_2\cdot\ldots\cdot\rho_{\nu(n)-1}+\ldots+\xi_{\nu(n)-1}$
is bounded in probability, that is, for any $a\in (0,1)$ there
exists $b=b(a)>0$ such that
\begin{equation}\label{bound}
\P\{\xi_1\rho_2\cdot\ldots\cdot\rho_{\nu(n)-1}+\ldots+\xi_{\nu(n)-1}>b\}\leq
a.
\end{equation}
Write, for any $x>0$,
\begin{eqnarray*}
&&\P\{\xi_1\rho_2\cdot\ldots\cdot\rho_{\nu(n)-1}+\ldots+\xi_{\nu(n)-1}>x\}\\&=&
\sum_{k\geq
2}\P\{\xi_1\rho_2\cdot\ldots\cdot\rho_{k-1}+\ldots+\xi_{k-1}>x,
S_{k-1}\leq n, S_{k-1}+\xi_k>n\}\\&=&\sum_{k\geq 2}\P\{\xi_1+
\xi_2\rho_1+\ldots+\xi_{k-1} \rho_1\cdot\ldots\cdot\rho_{k-2}>x,
S_{k-1}\leq n, S_{k-1}+\xi_k>n\}\\&=& \P\{\xi_1+
\xi_2\rho_1+\ldots+\xi_{\nu(n)-1}
\rho_1\cdot\ldots\cdot\rho_{\nu(n)-2}>x\}\\&\leq& \P\{\xi_1+
\xi_2\rho_1+\xi_3\rho_1\rho_2+\ldots>x\}.
\end{eqnarray*}
This proves \eqref{bound} (hence, \eqref{purp}) because $\E\log
\xi<\infty$ which is a consequence of ($\xi$) together with
$\E\log \rho\in [-\infty, 0)$ ensures that the series $\xi_1+
\xi_2\rho_1+\xi_3\rho_1\rho_2+\ldots$ converges a.s.

Part (ii) follows from \eqref{bound} and the observation
$$\frac{(n-S_{\nu(n)-1})\rho_{\nu(n)}}{(\P\{\xi>n\}^{-1/\alpha}}\leq \frac{\rho_{\nu(n)}}{n^{-1}\P\{\xi>n\}^{-1/\alpha}}=
\frac{\rho_{\nu(n)}}{c_2(n)}\topr 0, \quad n\to\infty,$$ where the
limit relation is guaranteed by Lemma \ref{lem:rhonu}.
\end{proof}
\begin{lem}\label{levymeasure}
Assume that $\beta\in (0,1)$. The measure $\mu$ defined in
\eqref{mu} satisfies
\begin{equation}\label{levy}
\int_{|{\bf x}|\neq 0}(|{\bf x}|\wedge 1)\mu({\rm d}{\bf
x})<\infty.
\end{equation}
In particular, $\mu$ is the L\'{e}vy measure of the
two-dimensional L\'{e}vy process ${\bf L}$ defined in \eqref{lt}.
\end{lem}
\begin{proof}
One part of \eqref{levy} is trivial: $$\int_{|{\bf x}|>1} \mu({\rm
d}{\bf x})= \mu\{(u,v)\in \mathbb{K}: u^2+v^2>1\}\leq
\mu\{(u,v)\in \mathbb{K}: u>2^{-1/2}~\text{or}~ v>
2^{-1/2}\}<\infty.$$ To prove the other part of \eqref{levy}, set,
for $n\in\N_0$, $A_n:=\{(u,v)\in\mathbb{K}:
u>2^{-n}~\text{or}~v>2^{-n}\}$. Now observe that
$(0,1]^2=\bigcup_{n\geq 1}(A_n\backslash A_{n-1})$ and that
$\mu(A_n)\leq \mathcal{C}_\mu 2^{n\beta/2}+2^{n\beta}$ for
$n\in\N_0$. Using these in combination with the inequality
$\sqrt{x_1^2+x_2^2}\leq \sqrt{2} (x_1\vee x_2)$ which holds for
nonnegative $x_1$ and $x_2$ we obtain
\begin{eqnarray*}
2^{-1/2}\int_{0<|{\bf x}|\leq 1}|{\bf x}|\mu({\rm d}{\bf x})&\leq&
2^{-1/2}\int_{(0,\,1]^2}|{\bf x}|\mu({\rm d}{\bf x})\leq
\sum_{n\geq 1}\int_{A_n\backslash A_{n-1}}(x_1\vee x_2)\mu({\rm
d}{\bf x})\\&\leq& \sum_{n\geq 1}2^{-(n-1)}\mu(A_n)\leq
\sum_{n\geq 1}2^{-(n-1)}(\mathcal{C}_\mu
2^{n\beta/2}+2^{n\beta})<\infty.
\end{eqnarray*}
To justify the last inequality we recall that $\beta\in (0,1)$.
\end{proof}

We are ready to prove the main results.
\begin{proof}[Proof of Theorem~\ref{thm:main11T}]
The transition from \eqref{Tn} to \eqref{Xn} is straightforward.
Hence, we only prove \eqref{Tn}. While either of the conditions
imposed on the distribution of $\rho$ ensures that $\E\log\rho\in
[-\infty, 0)$, condition ($\xi$) guarantees $\E\log\xi<\infty$.
This means that~\eqref{eq:right_transience} holds. Starting
with~\eqref{basic} we obtain a decomposition: for $n\in\N$,
\begin{eqnarray}
T_n&=&
S_{\nu(n)-1}+2\sum_{i=0}^{S_{\nu(n)-1}}U_i^{(S_{\nu(n)-1})}+(n-S_{\nu(n)-1})+2\sum_{i=S_{\nu(n)-1}+1}^n
U_i^{(n)}\notag\\&+&2\sum_{i=0}^{S_{\nu(n)-1}}
\big(U_i^{(n)}-U_i^{(S_{\nu(n)-1})}\big)+2\sum_{i<0}U_i^{(n)}.\label{decomp}
\end{eqnarray}
Since the random walk $X$ is transient to the right
(recall~\eqref{eq:right_transience}) the last summand is bounded
in probability as $n\to\infty$.

If condition $(\rho1)$ holds with $\alpha=\beta/2$, then part (i)
of Lemma~\ref{negligible(01)} applies with $\gamma>\beta/2$ as
defined in $(\rho1)$. If condition $(\rho2)$ holds with
$\beta/2\in \mathcal{I}$, then part (i) of
Lemma~\ref{negligible(01)} applies with any $\gamma>\beta/2$ such
that $\gamma\in \mathcal{I}$. In any event, we conclude that
relation \eqref{inter} holds. Thus, \eqref{Tn} is a consequence of
\begin{equation}\label{weak10}
n^{-2}\Big(S_{\nu(n)-1}+2\sum_{i=0}^{S_{\nu(n)-1}}U_i^{(S_{\nu(n)-1})}+(n-S_{\nu(n)-1})+2\sum_{i=S_{\nu(n)-1}+1}^n
U_i^{(n)}\Big)~\dod~2\chi, \quad n\to\infty.
\end{equation}
In view of $S_{\nu(n)-1}\leq n$ we have
\begin{equation}\label{under}
n^{-2} S_{\nu(n)-1}\topr 0,\quad n\to\infty.
\end{equation}
Recall the definition of $T^\prime_n$ given in \eqref{eq:tprime}.
We claim that, for $n\in\N$,
\begin{equation}\label{equ}
\begin{split}
\varrho_n&:=2\sum_{i=0}^{S_{\nu(n)-1}}{U_i^{(S_{\nu(n)-1})}}+(n-S_{\nu(n)-1})+2\sum_{i=S_{\nu(n)-1}+1}^n
U_i^{(n)}\\ &\od
2\sum_{k=1}^{S_{\nu(n)-1}}Z_k+T^\prime_{n-S_{\nu(n)-1}}:=\varpi_n
\end{split}
\end{equation}
which shows it is enough to prove
\begin{equation}\label{weak9}
n^{-2}\Big(2\sum_{k=1}^{S_{\nu(n)-1}}Z_k+T^\prime_{n-S_{\nu(n)-1}}\Big)~\dod~2\chi,\quad
n\to\infty.
\end{equation}
To check \eqref{equ}, we write, for $n\in\N$ and $x\geq 0$,
\begin{eqnarray*}
\P\{\varrho_n\leq x\} &=&\sum_{k\geq
1}\E\bigg[\P\bigg\{2\sum_{i=0}^{S_{k-1}}{U_i^{(S_{k-1})}}+(n-S_{k-1})+2\sum_{i=S_{k-1}+1}^n
U_i^{(n)}\leq x,\, \\ && S_{k-1}\leq n, S_k>n\big|(\xi_j,
\rho_j)_{1\leq j\leq k-1}\bigg\}\bigg]\\&=& \sum_{k\geq
1}\E\bigg[\P\bigg\{2\sum_{i=1}^{S_{k-1}}Z_i+T^\prime_{n-S_{k-1}}\leq
x,\, S_{k-1}\leq n, S_k>n|(\xi_j, \rho_j)_{1\leq j\leq
k-1}\bigg\}\bigg]\\ &=&\P\{\varpi_n\leq x\},
\end{eqnarray*}
where for the second equality we have used formula
\eqref{eq:cond}, the conditional independence of
$\sum_{i=0}^{S_{k-1}}U_i^{(S_{k-1})}$ and $\sum_{i=S_{k-1}+1}^n
U_i^{(n)}$, given $(\xi_j, \rho_j)_{1\leq j\leq k-1}$, and the
fact that, given $(\xi_j, \rho_j)_{1\leq j\leq k-1}$, the random
variable $n-S_{k-1}+2\sum_{i=S_{k-1}+1}^n U_i^{(n)}$ has the same
distribution as $T^{\prime}_{n-S_{k-1}}$, see \eqref{eq:dec2}.
Passing to the proof of~\eqref{weak9} we note that an appeal
to~\eqref{eq:main_two_sided_estimate} yields
\begin{equation}\label{eq:upper_and_lower_bounds}
\sum_{k=1}^{\tau^{\ast}_{\nu(n)-1}}\overline{\W}_{\tau_k}\leq
\sum_{k=1}^{S_{\nu(n)-1}} Z_k\leq
\sum_{k=1}^{\tau^{\ast}_{\nu(n)-1}+1} \overline{ \W}_{\tau_k}\quad
\P-\text{a.s.}
\end{equation}
Formula \eqref{weak9} holds provided that
\begin{equation}\label{weak11}
n^{-2}\Big(2\sum_{k=1}^{\tau^{\ast}_{\nu(n)-1}}\overline{\W}_{\tau_k}+T^\prime_{n-S_{\nu(n)-1}}\Big)~\dod~2\chi,\quad
n\to\infty
\end{equation}
and
\begin{equation*}
n^{-2}\Big(2\sum_{k=1}^{\tau^{\ast}_{\nu(n)-1}+1}\overline{\W}_{\tau_k}+T^\prime_{n-S_{\nu(n)-1}}\Big)~\dod~2\chi,\quad
n\to\infty
\end{equation*}
We shall only check \eqref{weak11}. The proof of the other limit
relation is analogous.

Recall that $a$ is a positive function satisfying
$\lim_{t\to\infty}t\P\{\xi_1>a(t)\}=1$. By Lemma \ref{future},
$\E\tau_1<\infty$ and $$\lim_{t\to\infty}
t\P\{S_{\tau_1}>a(t)x_1\}=(\E\tau_1)x_1^{-\beta},\quad x_1>0.$$
Further, parts (C2), (C3) and (C4) of Lemma \ref{prop:tail_main}
ensure
$$\lim_{t\to\infty}t\P\{\overline{\W}_{\tau_1}>a(t)^2x_2\}=(\E\tau_1)\mathcal{C}_\mu
x_2^{-\beta/2},\quad x_2>0.$$ These limit relations in combination
with Proposition \ref{prop:wp1} demonstrate that
$$\lim_{t\to\infty}t\P\{S_{\tau_1}>a(t)x_1
~\text{or}~\overline{\W}_{\tau_1}>a(t)^2x_2
\}=(\E\tau_1)\mu\{(u,v)\in\mathbb{K}: u>x_1~\text{or}~v>x_2\}$$
for all $x_1, x_2>0$, where $\mu$ is a measure defined in
\eqref{mu}. By Lemma 6.1 in \cite{Resnick:2007}, the latter
implies that
$$n\P\Big\{\Big(\frac{S_{\tau_1}}{a(n)}, \frac{\overline{\W}_{\tau_1}}{a(n)^2} \Big)\in \cdot\Big\}~\tov~ (\E\tau_1)\mu(\cdot),\quad
n\to\infty,$$ where $\tov$ denotes vague convergence in the set of
locally finite (Radon) measures on $\mathbb{K}$. By Theorem 4 in
\cite{Resnick+Greenwood:1979},
\begin{equation}\label{weak}
\Big(\frac{\sum_{k=1}^{[n\cdot]}(S_{\tau_k}-S_{\tau_{k-1}})}{a(n)},\frac{\sum_{k=1}^{[n\cdot]}\overline{\W}_{\tau_k}}{a(n)^2}
\Big)~\Rightarrow~ {\bf L}(h(\cdot)),\quad n\to\infty
\end{equation}
in the $J_1$-topology on $D^2$, where $h(t)=(\E\tau_1)t$ for
$t\geq 0$.

In view of $\E\tau_1<\infty$, $(\tau_n^\ast)_{n\in\N}$ is the
renewal process which corresponds to the finite mean standard
random walk $(\tau_k)_{k\in\N_0}$. According to the weak law of
large numbers for renewal processes $n^{-1}\tau_n^\ast\topr
(\E\tau_1)^{-1}$ as $n\to\infty$. It is well-known that this can
be strengthened to
\begin{equation}\label{weak3}
\sup_{t\in [0,T]}|n^{-1}\tau_{[nt]}^\ast-(\E\tau_1)^{-1}t|~\topr~
0,\quad n\to\infty
\end{equation}
for all $T>0$ or equivalently
\begin{equation}\label{weak2}
n^{-1}\tau^\ast_{[n\cdot]}~{\overset{{\rm
J_1}}{\Rightarrow}}~ g(\cdot),\quad n\to\infty,
\end{equation}
where $g(t)=(\E\tau_1)^{-1}t$ for $t\geq 0$. Since the limit is
deterministic, \eqref{weak} and \eqref{weak2} can be combined into
the joint convergence
\begin{equation}\label{inter2}
\Big(\Big(\frac{\sum_{k=1}^{[n\cdot]}(S_{\tau_k}-S_{\tau_{k-1}})}{a(n)},\frac{\sum_{k=1}^{[n\cdot]}\overline{\W}_{\tau_k}}{a(n)^2}\Big),
\frac{\tau^\ast_{[n\cdot]}}{n} \Big)~\Rightarrow~ \big({\bf
L}(h(\cdot)), g(\cdot)\big),\quad n\to\infty
\end{equation}
in the product $J_1$-topology on $D^2\times D$. Furthermore, since
the convergence in \eqref{weak3} is uniform, then passing to
versions of
$\Big(\frac{\sum_{k=1}^{[n\cdot]}(S_{\tau_k}-S_{\tau_{k-1}})}{a(n)},\frac{\sum_{k=1}^{[n\cdot]}\overline{\W}_{\tau_k}}{a(n)^2}\Big)$
and $\frac{\tau^\ast_{[n\cdot]}}{n}$ which converge $\P$-a.s.\ we
infer that we can use the same homeomorphisms $\lambda_n(t)$ which
appear in the definition of the $J_1$-convergence for both terms.
This shows that \eqref{inter2} holds in the $J_1$-topology on
$D^3$. An application of Lemma \ref{whi} together with the
continuous mapping theorem yields
\begin{equation}\label{weak4}
\Big(\frac{\sum_{k=1}^{\tau^\ast_{[n\cdot]}}(S_{\tau_k}-S_{\tau_{k-1}})}{a(n)},\frac{\sum_{k=1}^{\tau^\ast_{[n\cdot]}}\overline{\W}_{\tau_k}}{a(n)^2}\Big)~\Rightarrow~
{\bf L}(\cdot),\quad n\to\infty
\end{equation}
and
\begin{equation}\label{weak415}
\Big(\frac{\sum_{k=1}^{\tau^\ast_{[n\cdot]}+1}(S_{\tau_k}-S_{\tau_{k-1}})}{a(n)},\frac{\sum_{k=1}^{\tau^\ast_{[n\cdot]}}\overline{\W}_{\tau_k}}{a(n)^2}\Big)~\Rightarrow~
{\bf L}(\cdot),\quad n\to\infty
\end{equation}
in the $J_1$-topology on $D^2$. Since the random walk
$(S_n)_{n\in\N_0}$ is $\P$-a.s.\ nondecreasing and
$\tau_{\tau^\ast_n}\leq n\leq \tau_{\tau^\ast_n+1}$ $\P$-a.s.,
relations \eqref{weak4} and \eqref{weak415} entail
\begin{equation}\label{weak100}
\Big(\frac{\sum_{k=1}^{[n\cdot]}\xi_k}{a(n)},\frac{\sum_{k=1}^{\tau^\ast_{[n\cdot]}}\overline{\W}_{\tau_k}}{a(n)^2}\Big)~\Rightarrow~
{\bf L}(\cdot),\quad n\to\infty
\end{equation}
in the $J_1$-topology on $D^2$.

An argument leading to Theorem 3.6 in\footnote{The independence
assumption imposed in the cited result is only made to ensure a
limit relation like \eqref{weak100}. The passage from
\eqref{weak100} to \eqref{weak5} is justified by the continuous
mapping theorem and as such does not require the aforementioned
independence.} \cite{Straka+Henry:2011} enables us to conclude
that the last limit relation entails
\begin{equation}\label{weak5}
\Big(\frac{\sum_{k=1}^{\nu([n\cdot])-1} \xi_k}{n},
\frac{\sum_{k=1}^{\tau^\ast_{\nu([n\cdot])-1}}\overline{\W}_{\tau_k}}{n^2}\Big)~\Rightarrow~\big(((L_1\circ
L_1^\leftarrow)(\cdot)-)^+, (((L_2\circ L_1^\leftarrow)
(\cdot)-))^+\big),\quad n\to\infty
\end{equation}
in the $J_1$-topology on $D^2$, where we write $(X(t)^+)$ for
$(X(t+))$.

By Lemma 4.2 (iv) in \cite{Straka+Henry:2011} the limit process in
\eqref{weak5} admits no fixed discontinuities. In view of this we
obtain
\begin{equation*}
\Big(\frac{\sum_{k=1}^{\nu(n)-1}\xi_k}{n},
\frac{\sum_{k=1}^{\tau^\ast_{\nu(n)-1}}\overline{\W}_{\tau_k}}{n^2}\Big)~\dod~(L_1(L_1^\leftarrow(1)-),
L_2(L_1^\leftarrow(1)-)),\quad n\to\infty
\end{equation*}
as a consequence of \eqref{weak5}. An application of \eqref{weak8}
yields
\begin{equation}\label{weak7}
\Big(\frac{S_{\nu(n)-1}}{n},
\frac{\sum_{k=1}^{\tau^\ast_{\nu(n)-1}}\overline{\W}_{\tau_k}}{n^2},
\frac{T_n^\prime}{n^2}\Big)~\dod~\big(L_1(L_1^\leftarrow(1)-),
L_2(L_1^\leftarrow(1)-), M(1)\big),\quad n\to\infty
\end{equation}
having utilized the fact that $T_n^\prime$ is independent of the
other components on the left-hand side. In view of
$n-S_{\nu(n)-1}\topr \infty$ as $n\to\infty$ this implies that, as
$n\to\infty$,
\begin{equation}\label{weak234}
\Big(\frac{n-S_{\nu(n)-1}}{n},
\frac{\sum_{k=1}^{\tau^\ast_{\nu(n)-1}}\overline{\W}_{\tau_k}}{n^2},
\frac{T_{n-S_{\nu(n)-1}}^\prime}{(n-S_{\nu(n)-1})^2}\Big)~\dod~\big(1-L_1(L_1^\leftarrow(1)-),
L_2(L_1^\leftarrow(1)-), M(1)\big)
\end{equation}
and thereupon
$$\frac{2\sum_{k=1}^{\tau^\ast_{\nu(n)-1}}\overline{\W}_{\tau_k}+
T^\prime_{n-S_{\nu(n)-1}}}{n^2}~\dod~
2L_2(L_1^\leftarrow(1)-)+M(1)(1-L_1(L_1^\leftarrow(1)-))^2\od
2\chi,\;\; n\to\infty.$$ Thus, relation \eqref{weak11} holds true.
\end{proof}

\begin{proof}[Proof of Theorem~\ref{thm:main2T}]
Relation \eqref{Xn2} is an immediate consequence of \eqref{Tn2}.
Therefore, we only focus on \eqref{Tn2}. Its proof proceeds along
the lines of the proof of Theorem \ref{thm:main11T} but is much
simpler. In view of this, we only give a sketch.

We shall use decomposition \eqref{decomp}. We already know from
the proof of Theorem \ref{thm:main11T} that the last summand in
\eqref{decomp} is bounded in probability. Further, note that under
the assumptions of Theorem \ref{thm:main2T} all the conditions of
part (ii) Lemma \ref{negligible(01)} are met. Thus,
\begin{equation}\label{inter1}
\P\{\xi>n\}^{1/\alpha}\sum_{j=0}^{S_{\nu(n)-1}}\Big(U_j^{(n)}-U_j^{(S_{\nu(n)-1})}\Big)~\topr~0,\quad
n\to\infty.
\end{equation}
Also,
$$\frac{S_{\nu(n)-1}}{\P\{\xi>n\}^{-1/\alpha}}~\topr~0,\quad n\to\infty$$
because $S_{\nu(n)-1}\leq n$ and the denominator varies regularly
of index $\beta/\alpha\geq 2$. In view of these limit relations,
\eqref{Tn2} is a consequence of
\begin{equation*}
\P\{\xi>n\}^{1/\alpha}\Big(2\sum_{i=0}^{S_{\nu(n)-1}}U_i^{(S_{\nu(n)-1})}+(n-S_{\nu(n)-1})+2\sum_{i=S_{\nu(n)-1}+1}^n
U_i^{(n)}\Big)~\dod~2\widehat{L}_2(\widehat{L}_1^\leftarrow(1))
\end{equation*}
as $n\to\infty$, which in its turn is implied by
$$\P\{\xi>n\}^{1/\alpha}\Big(2\sum_{k=1}^{\tau^\ast_{\nu(n)-1}}\overline{\W}_{\tau_k}+ T^\prime_{n-S_{\nu(n)-1}}\Big)~\dod~2\widehat{L}_2(\widehat{L}_1^\leftarrow(1)-)\od
2\widehat{L}_2(\widehat{L}_1^\leftarrow(1)),\quad n\to\infty$$ by
the same reasoning as given in the proof of
Theorem~\ref{thm:main11T}. Condition $(\rho1)$ ensures that
$\E\log\rho\in [-\infty, 0)$. By Lemma \ref{future},
$\E\tau_1<\infty$ and $$\lim_{t\to\infty}
t\P\{S_{\tau_1}>a(t)x_1\}=(\E\tau_1)x_1^{-\beta},\quad x_1>0.$$
According to part (C1) of Lemma
\ref{prop:tail_main},$$\lim_{t\to\infty}t\P\{\overline{\W}_{\tau_1}>t^{1/\alpha}x_2\}=(\E\tau_1)\mathcal{C}_Z(\alpha)
x_2^{-\alpha},\quad x_2>0.$$ Observe that the cited result applies
in the case $\alpha\in (0, \beta/2)$ in view of
$\E\xi^{2\alpha}<\infty$ which is secured by $(\xi)$. These limit
relations in combination with Proposition \ref{prop:w1}
demonstrate that $$\lim_{t\to\infty}t\P\{S_{\tau_1}>a(t)x_1
~\text{or}~\overline{\W}_{\tau_1}>t^{1/\alpha} x_2
\}=(\E\tau_1)(x_1^{-\beta}+{\mathcal{C}_Z(\alpha)}
x_2^{-\alpha})$$ for all $x_1, x_2>0$. Arguing as in the proof of
Theorem \ref{thm:main11T} we arrive at counterparts of
\eqref{weak}, \eqref{weak100} and \eqref{weak234}, respectively,
\begin{equation*}
\Big(\frac{\sum_{k=1}^{[n\cdot]}(S_{\tau_k}-S_{\tau_{k-1}})}{a(n)},\frac{\sum_{k=1}^{[n\cdot]}\overline{\W}_{\tau_k}}{n^{1/\alpha}}
\Big)~\Rightarrow~ (\widehat{L}_1(h(\cdot),
\widehat{L}_2(h(\cdot))),\quad n\to\infty
\end{equation*}
in the $J_1$-topology on $D^2$, where, as before,
$h(t)=(\E\tau_1)t$ for $t\geq 0$;
\begin{equation}\label{weak100000}
\Big(\frac{\sum_{k=1}^{[n\cdot]}\xi_k}{a(n)},\frac{\sum_{k=1}^{\tau^\ast_{[n\cdot]}}\overline{\W}_{\tau_k}}{n^{1/\alpha}}\Big)~\Rightarrow~
(\widehat{L}_1(\cdot), \widehat{L}_2(\cdot)),\quad n\to\infty
\end{equation}
in the $J_1$-topology on $D^2$; as $n\to\infty$,
$$\Big(\frac{n-S_{\nu(n)-1}}{n}, \frac{\sum_{k=1}^{\tau^\ast_{\nu(n)-1}}\overline{\W}_{\tau_k}}{\P\{\xi>n\}^{-1/\alpha}},
\frac{T_{n-S_{\nu(n)-1}}^\prime}{(n-S_{\nu(n)-1})^2}\Big)\dod\big(1-\widehat{L}_1(\widehat{L}_1^\leftarrow(1)-),
\widehat{L}_2(\widehat{L}_1^\leftarrow(1)-), M(1)\big).$$ Since
$\P\{\xi>n\}^{-1/\alpha}\sim n^{\beta/\alpha}\ell(n)^{-1/\alpha}$
we infer $\lim_{n\to\infty}n^{-2}\P\{\xi>n\}^{-1/\alpha}=\infty$
and thereupon
$$\P\{\xi>n\}^{1/\alpha}\Big(2\sum_{k=1}^{\tau^\ast_{\nu(n)-1}}\overline{\W}_{\tau_k}+
T^\prime_{n-S_{\nu(n)-1}}\Big)~\dod~
2\widehat{L}_2(\widehat{L}_1^\leftarrow(1)-)\od
2\widehat{L}_2(\widehat{L}_1^\leftarrow(1)),\quad n\to\infty,$$
where the last distributional equality is implied by the
independence.
\end{proof}

\subsection{The case $\beta=1$}

We start by proving several auxiliary results. {The
functions $a$, $\pi$ and $\pi^\ast$ appearing below are defined in
Section \ref{beta1}.}
\begin{lem}\label{lem:lln_renewal_if_beta=1}
Assume that ($\xi$) holds for $\beta=1$. Then, for every $T>0$,
$$\sup_{u\in
[0,\,T]}\left|\frac{\nu(tu)}{t\pi^{\ast}(t)}-u\right|~\topr~
0,\quad t\to\infty.$$ In particular,
\begin{equation}\label{eq:s_nu_beta_1}
\frac{S_{\nu(t)}}{t}~\topr~ 1,\quad t\to\infty.
\end{equation}
\end{lem}
\begin{proof}
It is known {(see, for instance, Example 2 on
p.~1034 in \cite{Haan+Resnick:1979})} that if $t\mapsto
\P\{\xi>t\}$ is regularly varying at $\infty$ of index $-1$, then
$$\frac{S_{[t\cdot]}-t(\cdot)\pi(t)}{a(t)}~{\overset{{\rm
J_1}}{\Rightarrow}}~ Y_1(\cdot),\quad t\to\infty$$ in the
$J_1$-topology on $D$, where $(Y_1(u))_{u\geq 0}$ is a $1$-stable
spectrally positive L\'{e}vy process. This yields
\begin{equation}\label{eq:inter}
\frac{S_{[t\cdot ]}}{t\pi(t)}~{\overset{{\rm
J_1}}{\Rightarrow}}~ f(\cdot),\quad t\to\infty,
\end{equation}
where $f(u)=u$ for $u\geq 0$. Using continuity of the inversion
(see \cite{Whitt:1971}) we deduce
$$\frac{\nu(t(\cdot)\pi(t))}{t}=\inf\left\{s\geq
0:\frac{S_{[ts]}}{t\pi(t)}>(\cdot)\right\}~{\overset{{\rm
M_1}}{\Rightarrow}}~ f(\cdot),\quad t\to\infty.$$ Since the limit
is continuous, the convergence is actually locally uniform, that
is, for every $T>0$,
$$\sup_{u\in [0,\,T]}\left|\frac{\nu(tu
\pi(t))}{t}-u\right|~\topr~0,\quad t\to\infty.$$ Replacing $t$
with $t\pi^{\ast}(t)$ and using
$\lim_{t\to\infty}\pi^{\ast}(t)\pi(t\pi^{\ast}(t))=1$ we obtain
the first claim of the lemma.

Relation \eqref{eq:s_nu_beta_1} follows from \eqref{eq:inter} with
$t\pi^\ast(t)$ replacing $t$ in combination with the first claim
and Lemma \ref{whi} with $k=1$.
\end{proof}

\begin{lem}\label{lem:overshot_if_beta=1}
Assume that ($\xi$) holds for $\beta=1$. Then
$$\frac{n-S_{\nu(n)-1}}{a(n\pi^{\ast}(n))}~\topr~ 0,\quad n\to\infty.$$
\end{lem}
\begin{proof}
Fix $\delta>0$ and write $$\P\{n-S_{\nu(n)-1}\geq \delta
a(n\pi^{\ast}(n))\}\leq
\P\left\{\frac{m(n-S_{\nu(n)-1})}{m(n)}\geq \frac{m(\delta
a(n\pi^{\ast}(n)))}{m(n)}\right\}.$$ By Theorem 6 in
\cite{Erickson:1970}, $$\lim_{n\to\infty}
\P\left\{\frac{m(n-S_{\nu(n)-1})}{m(n)}\leq x\right\}=x,\quad x\in
[0,1].$$ Thus, it is enough to show that
$$\liminf_{n\to\infty}\frac{m(\delta a(n\pi^{\ast}(n)))}{m(n)}\geq
1.$$ Since $m$ is slowly varying, this limit relation is
equivalent to
$$\liminf_{n\to\infty}\frac{m(a(n\pi^{\ast}(n)))}{m(n)}=
\liminf_{n\to\infty}\frac{\pi(n\pi^{\ast}(n))}{m(n)}=\liminf_{n\to\infty}\frac{1}{\pi^{\ast}(n)m(n)}\geq
1.$$ The last inequality follows by an application of Fatou's
lemma together with Lemma \ref{lem:lln_renewal_if_beta=1}:
$$1=\E \lim_{n\to\infty}\frac{\nu(n)}{n\pi^{\ast}(n)}\leq
\liminf_{n\to\infty}\frac{U(n)}{n\pi^{\ast}(n)}=\liminf_{n\to\infty}\frac{U(n)m(n)}{n}\frac{1}{m(n)\pi^{\ast}(n)}
=\liminf_{n\to\infty}\frac{1}{m(n)\pi^{\ast}(n)},$$ where the
first limit is understood as the limit in probability, $U$ denotes
the renewal function (see \eqref{eq:RenFunDf}), and the last
equality follows from \eqref{ren1}.
\end{proof}

Lemma \ref{negligible1} given next is a counterpart of Lemma
\ref{negligible(01)}.
\begin{lem}\label{negligible1}
Assume that $\E\log\rho\in [-\infty, 0)$ and that ($\xi$) holds
{for $\beta=1$}.
\begin{itemize}
\item[(i)] If  ($\xi\rho 1$) holds, and $\E\rho^\gamma<\infty$ for some $\gamma>1/2$, then
\begin{equation}\label{inter_beta_1}
a(n\pi^{\ast}(n))^{-2}Y_n ~\topr~0,\quad n\to\infty.
\end{equation}
\item[(ii)] If  ($\xi\rho2$) holds for some $\alpha\leq 1/2$ and $\E\rho^\gamma<\infty$ for some $\gamma>\alpha$, then
\begin{equation}\label{inter001}
(n\pi^{\ast}(n))^{-1/\alpha}Y_n~\topr~0,\quad n\to\infty.
\end{equation}
\end{itemize}
\end{lem}
The proof is omitted, for it follows along the same lines as the
proof of Lemma~\ref{negligible(01)}.

\begin{proof}[Proof of Theorem \ref{thm:main3T}]
Recall that $\beta=1$. If condition $(\rho1)$ holds with
$\alpha=\beta/2=1/2$, then Lemma~\ref{negligible1}(i) applies with
$\gamma>1/2$ as defined in $(\rho1)$. If condition $(\rho2)$ holds
with $1/2\in \mathcal{I}$, then  Lemma \ref{negligible1}(i)
applies with any $\gamma>1/2$ such that $\gamma\in \mathcal{I}$.
Thus, in any event, \eqref{inter_beta_1} holds.

Using once again decomposition \eqref{decomp} we conclude that
\eqref{Tn_beta_1} is a consequence of
\begin{equation}\label{weak10_beta_1}
\frac{\Big(S_{\nu(n)-1}+2\sum_{i=0}^{S_{\nu(n)-1}}U_i^{(S_{\nu(n)-1})}+(n-S_{\nu(n)-1})+2\sum_{i=S_{\nu(n)-1}+1}^n
U_i^{(n)}\Big)}{a(n\pi^{\ast}(n))^2}~\dod~ 2L_2(1)
\end{equation}
as $n\to\infty$. In view of $S_{\nu(n)-1}\leq n$ $\P$-a.s.,
\begin{equation}\label{under_beta_1}
\frac{S_{\nu(n)-1}}{a(n\pi^{\ast}(n))^2}\topr 0,\quad n\to\infty
\end{equation}
because the denominator is regularly varying at $\infty$ of index
$2$.

From the proof of Theorem \ref{thm:main11T} we know that
distributional equality \eqref{equ} holds. Hence, it is enough to
show that
\begin{equation}\label{weak9_beta_1}
\frac{\sum_{k=1}^{S_{\nu(n)-1}}Z_k}{a(n\pi^{\ast}(n))^2}~\dod~L_2(1),\quad
n\to\infty
\end{equation}
and
\begin{equation}\label{weak9_beta_11}
\frac{T^\prime_{n-S_{\nu(n)-1}}}{a(n\pi^{\ast}(n))^2}~\topr~0,\quad
n\to\infty.
\end{equation}
We first prove~\eqref{weak9_beta_11}. Using \eqref{weak8} in
combination with $n-S_{\nu(n)-1} \topr +\infty$ as $n\to\infty$
and the independence of $(T_k^\prime)_{k\in\N_0}$ and
$n-S_{\nu(n)-1}$ we infer
$$\frac{T^{\prime}_{n-S_{\nu(n)-1}}}{(n-S_{\nu(n)-1})^2}~\dod~ {2} \vartheta,\quad n\to\infty.$$
With this at hand, \eqref{weak9_beta_11} follows from Lemma
\ref{lem:overshot_if_beta=1}.

In order to prove \eqref{weak9_beta_1} note that in formula
\eqref{weak4} we still have convergence of the second components,
that is,
$$\frac{\sum_{k=1}^{\tau^\ast_{[n\cdot]}}\overline{\W}_{\tau_k}}{a(n)^2}~{\overset{{\rm
J_1}}{\Rightarrow}}~ L_2(\cdot),\quad n\to\infty$$ or,
equivalently,
$$\frac{\sum_{k=1}^{\tau^\ast_{[n\pi^{\ast}(n)(\cdot)]}}\overline{\W}_{\tau_k}}{a(n\pi^{\ast}(n))^2}~{\overset{{\rm
J_1}}{\Rightarrow}}~ L_2(\cdot),\quad n\to\infty.$$ By Lemma
\ref{lem:lln_renewal_if_beta=1},
$$\frac{\nu(n\cdot)-1}{n\pi^{\ast}(n)}~{\overset{{\rm
J_1}}{\Rightarrow}}~ f(\cdot),\quad n\to\infty,$$ where $f(t)=t$
for $t\geq 0$. Using once again Lemma \ref{whi} with $k=1$ we
infer
$$\frac{\sum_{k=1}^{\tau^\ast_{\nu(n)-1}}\overline{\W}_{\tau_k}}{a(n\pi^{\ast}(n))^2}~\dod~L_2(1),\quad
n\to\infty.$$ The same limit relation holds with $\nu(n)$
replacing $\nu(n)-1$. In view of \eqref{eq:upper_and_lower_bounds}
we arrive at \eqref{weak9_beta_1}.
\end{proof}

\begin{proof}[Proof of Theorem~\ref{thm:main3T1} ]
As before, we only focus on the formula involving $T_n$, that is,
\eqref{Tn2_beta_1}. The proof {of \eqref{Tn2_beta_1}} is similar
to but much simpler than the proof of Theorem \ref{thm:main2T}. In
view of this, we only give a sketch.

According to the proof of Theorem \ref{thm:main2T} the last
summand in \eqref{decomp} is bounded in probability. Under the
assumptions of Theorem \ref{thm:main3T1} the conditions of Lemma
\ref{negligible1}(ii) are satisfied, whence
\begin{equation}\label{inter1_beta_1}
(n\pi^{\ast}(n))^{-1/\alpha}\sum_{j=0}^{S_{\nu(n)-1}}\Big(U_j^{(n)}-U_j^{(S_{\nu(n)-1})}\Big)~\topr~0,\quad
n\to\infty
\end{equation}
Further,
$$\frac{S_{\nu(n)-1}}{(n\pi^{\ast}(n))^{1/\alpha}}~\topr~0,\quad n\to\infty$$
because $S_{\nu(n)-1}\leq n$ $\P$-a.s.\ and the denominator is
regularly varying at $\infty$ of index $1/\alpha \geq 2$. In view
of \eqref{weak100000}
\begin{equation*}
\frac{\sum_{k=1}^{\tau^\ast_{[n\pi^{\ast}(n)(\cdot)]}}\overline{\W}_{\tau_k}}{(n\pi^{\ast}(n))^{1/\alpha}}~{\overset{{\rm
J_1}}{\Rightarrow}}~L_2(\cdot), \quad n\to\infty.
\end{equation*}
As in the proof of Theorem \ref{thm:main3T}, an appeal to
$J_1$-continuity of the composition and Lemma
\ref{lem:lln_renewal_if_beta=1} enables us to conclude that
$$\frac{\sum_{k=1}^{\tau^\ast_{\nu(n)-1}}\overline{\W}_{\tau_k}}{(n\pi^{\ast}(n))^{1/\alpha}}~{\overset{{\rm
J_1}}{\Rightarrow}}~L_2(\cdot), \quad n\to\infty,$$ and that its
counterpart holds with $\nu(n)$ replacing $\nu(n)-1$. Finally, we
claim that
$$\frac{T^{\prime}_{n-S_{\nu(n)-1}}}{(n\pi^{\ast}(n))^{1/\alpha}}~\topr~0,\quad n\to\infty.$$ Indeed, this is a consequence of
\eqref{weak8}, Lemma \ref{lem:overshot_if_beta=1} and the limit
relation
$$\lim_{n\to\infty}\frac{a(n\pi^{\ast}(n))^2}{(n\pi^{\ast}(n))^{1/\alpha}}=0$$ that we are now going to prove. When $\alpha<1/2$, the latter holds,
for the function $a$ is then regularly varying at $\infty$ of
index $1<1/(2\alpha)$. When $\alpha=1/2$, we have
$\lim_{t\to\infty}\ell(t)=0$ by assumption. This entails
$a(t)=o(t)$ as $t\to\infty$, and the limit relation follows.
\end{proof}

\section*{Acknowledgment}

The authors thank Vitali Wachtel for bringing the article
\cite{Korshunov:2009} to their attention. D. Buraczewski and P.
Dyszewski were partially supported by the National Science Center,
Poland (Sonata Bis, grant number DEC-2014/14/E/ST1/00588).

\addcontentsline{toc}{section}{References}
\bibliographystyle{plain}

\vspace{1cm}

\footnotesize

\textsc{Dariusz Buraczewski and Piotr Dyszewski:} Mathematical
Institute, University of Wroclaw, 50-384 Wroclaw, Poland\\
\textit{E-mail}: \texttt{dbura@math.uni.wroc.pl;\
pdysz@math.uni.wroc.pl}

\bigskip

\textsc{Alexander Iksanov and Alexander Marynych:} Faculty of Computer Science and Cybernetics, Taras Shev\-chen\-ko National University of Kyiv,
01601 Kyiv, Ukraine\\
\textit{E-mails}: \texttt{iksan@univ.kiev.ua; \
marynych@unicyb.kiev.ua}

\end{document}